\documentclass[a4paper,10pt]{article}

\usepackage{amsthm}
\usepackage{amsmath}
\usepackage{amssymb}
\usepackage{dsfont}
\usepackage[pdftex]{color, graphicx}
\usepackage{subfigure}
\usepackage{enumerate}

\usepackage[pdfauthor={Myriam Finster},
	colorlinks=true,
	linkcolor=black,
	citecolor=black,
	urlcolor=black,
	filecolor=black,
	pdftitle={A series of coverings of the regular n-gon},
	pdftex]{hyperref}
\usepackage{url}

\def\id{\textrm{id}}

\def\mmod{\, \textrm{mod} \,}

\def\pred{\mathrm{pred}}
\def\suc{\mathrm{suc}}

\def\PSL{\mathrm{PSL}_2(\R)}
\def\GL{\mathrm{GL}_2(\R)}
\def\SL{\mathrm{SL}_2(\R)}
\def\SO{\mathrm{SO}_2(\R)}

\def\R{\mathds{R}}
\def\Z{\mathds{Z}}
\def\HH{\mathds{H}}

\def\bs{\backslash}

\newtheorem{lemma}{Lemma}[section]

\newtheorem{corollary}[lemma]{Corollary}

\newtheorem*{ThmOdd}{Theorem 1 a)}
\newtheorem*{ThmEven}{Theorem 1 b)}
\newtheorem*{ThmInf}{Theorem 2}
\newtheorem*{CorEinleitung}{Corollary}

\theoremstyle{definition}

\newtheorem*{claim}{Claim}

\newtheorem{remark}[lemma]{Remark}

\input xy
\xyoption{all}

\title{A series of coverings of the regular n-gon}
\author{Myriam Finster \\ \\ \normalsize Karlsruhe Institute of Technology (KIT) \\ \normalsize e-mail: myriam.finster@kit.edu}

\begin{document}

\maketitle

\begin{abstract}
 We define an infinite series of translation coverings of Veech's double-n-gon for odd $n \geq 5$ which share the same Veech group. Additionally we give an infinite series of translation coverings with constant Veech group of a regular n-gon for even $n \geq 8$. These families give rise to explicit examples of infinite translation surfaces with lattice Veech group.
\end{abstract}

\section{Introduction}
\label{sec:into}
Billiards in a polygon with angles commensurable with $\pi$, can be investigated by considering the geodesic flow on an appropriate translation surface (see \cite{ZK75}). 
The fundamental work of Veech in 1989 connects the properties of the geodesic flow to the so called \textit{Veech group} of the \textit{translation surface} (see \cite{Vee89}). 
The \textit{Veech alternative} states, that if the Veech group is a lattice, then the billiard flow in each direction is either periodic or uniquely ergodic.

Some translation surfaces which have the lattice property already considered by Veech himself, are the double-$n$-gons $X_n$ for odd $n \geq 5$. They arise from billiards in a triangle with angles $\pi/n$, $\pi/n$ and $(n-2)\pi/n$. For even $n$, $n \geq 8$, the regular double-$n$-gons considered by Veech, are degree $2$ coverings of the regular $n$-gons $X_n$. 
The Veech groups of all these surfaces are well known lattice groups.
\\ 

For every odd $n \geq 5$ and every $d \geq 2$ we define a \textit{translation covering}
$p_{n,d}: Y_{n,d} \to X_n$ of degree $d$ and calculate the Veech group of $Y_{n,d}$.

\begin{ThmOdd}
For odd $n\geq5$, the Veech group of $Y_{n,d}$ is
 $$\Gamma_n := \Gamma(Y_{n,d}) = \langle -I, T_n, R_n \, {T_n}^2 \, {R_n}^{-1}, \dots, {R_n}^{\frac{n-1}{2}} \, {T_n}^2 \, {R_n}^{-\frac{n-1}{2}} \rangle\;,$$
where
$$R_n = \begin{pmatrix} \cos{\frac{\pi}{n}} & -\sin{\frac{\pi}{n}} \\ \sin{\frac{\pi}{n}} & \cos{\frac{\pi}{n}} \end{pmatrix} \; , \quad T_n = \begin{pmatrix} 1 & \lambda_n \\ 0 & 1 \end{pmatrix} \; , \quad \lambda_n= 2 \cot{\frac{\pi}{n}}$$
and $I \in \GL$ is the identity matrix.
\end{ThmOdd}

If we likewise define translation coverings $p_{n,d}: Y_{n,d} \to X_n$ in the even case for $n \geq 8$ and $d \geq 2$, we get a similar statement.

\begin{ThmEven}
For even $n \geq 8$, the Veech group of $Y_{n,d}$ is
 $$\begin{array}{rl}
    \Gamma_n :=& \Gamma(Y_{n,d})\\
    = &\langle -I, T_n, {R_n}^2 \, {T_n}^2 \, {R_n}^{-2}, \dots, {R_n}^{n-2} \, {T_n}^2 \, {R_n}^{-(n-2)}, \, ({T_n}^{-1}{R_n}^2)^2, \\ \noalign{\smallskip}
     & {R_n}^2 \,({T_n}^{-1}{R_n}^2)^2 \, {R_n}^{-2}, \dots, {R_n}^{n-2} \, ({T_n}^{-1}{R_n}^2)^2 \, {R_n}^{-(n-2)} \rangle\;,
   \end{array} $$
where $R_n$, $T_n$, $\lambda_n$ and $I$ are defined as in Theorem 1 a).
\end{ThmEven}

So for all $n \geq 5$, $n \neq6$ we obtain an infinite family of translation coverings with fixed Veech group $\Gamma_n$.  
The \textit{Teichm\"uller curve} arising from the translation surface $Y_{n,d}$ has the following properties.

\begin{CorEinleitung}
 $\HH / \Gamma_n$ has genus $0$, $\frac{n+1}{2}$ cusps if $n$ is odd and $\frac{n+2}{2}$ cusps if $n$ is even.
\end{CorEinleitung}

The limit of each family is an infinite translation surface, $Y_{n,\infty}$.

\begin{ThmInf}
 For $n \geq 5$, $n \neq 6$, the Veech group $\Gamma(Y_{n, \infty})$ of $Y_{n,\infty}$ is $\Gamma_n$. In particular $Y_{n,\infty}$ is an infinite translation surface with a lattice Veech group.
\end{ThmInf}

So the families give rise to explicit examples of infinite translation surfaces with lattice Veech groups, which are rare by now. One example of an infinite translation surface with lattice Veech group was found in \cite{HoWe09}. It is a $\Z$-cover of a double cover of the regular octagon. A second example is the translation surface described in \cite{Ho08}. It is build from two infinite polygons, the convex hull of the points $(n, n^2)$ and the convex hull of the points $(n, -n^2)$.
Another example is the infinite staircase origami $Z_{2,0}^\infty$, calculated in \cite{SchHu09}.\\

We will treat the case where $n$ is odd in detail. The proof for even $n$ works very similar, hence we keep it short.

\subsection*{Acknowledgements}

I would like to thank Frank Herrlich for his continuing support, Gabriela Schmit\-h\"usen for many helpful discussions and inspiring suggestions and Piotr Przytycki for useful comments. 
This work was partially supported by the Landesstiftung Baden-W\"urttemberg within the project ``With origamis to Teichm\"uller curves in moduli space''.

\section{Basic definitions and preliminaries.}
In this section we want to shortly review the basic definitions used to state and prove our theorems. For a more detailed introduction to Veech groups see e.g.\! \cite{HS01} or \cite{Vor96}. More details on Teichm\"uller curves can be found e.g.\! in \cite{HSch07}.\\

A \textit{(finite) translation surface} $X$ is a compact Riemann surface with a finite, non\-empty set $\Sigma(X)$ of \textit{singular points} together with a maximal $2$-dimensional atlas $\omega$ on $X \setminus \Sigma(X)$, such that all transition functions between the charts are translations. The $2$-dimensional atlas $\omega$ induces a flat metric on $X \setminus \Sigma(X)$, whereas 
the angles around points in $\Sigma(X)$ are integral multiples of $2 \pi$. If the angle around a singular point is $2 \pi$, then the flat metric can be extended to that point, otherwise the flat metric has a conical singularity.
A translation surface is obtained by gluing finitely many polygons via identification of edge pairs using translations. In this construction, singularities may arise from the vertices of the polygons. Especially in this situation the translation structure on $X$ is somehow obvious, so we often omit $\omega$ in the notation.\\

A continuous map $p:Y \to X$ is called a \textit{translation covering}, if $p(\Sigma(Y)) = \Sigma(X)$ and $p:Y \setminus \Sigma(Y) \to X \setminus \Sigma(X)$ is locally a translation. Since both $X$ and $Y$ are compact, $p$ is a finite covering map in the topological sense, ramified at most over the singularities $\Sigma(X)$. In the context of translation coverings $p: Y \to X$, we call $X$ the \textit{base surface} and $Y$ the \textit{covering surface} of $p$. A translation surface $(X,\omega)$ that, in this sense, is not the covering surface of a base surface with smaller genus, is called \textit{primitive}. In the following, only base surfaces of genus greater than one are considered and all singularities have conical angle greater than $2 \pi$, so $p^{-1}(\Sigma(X)) = \Sigma(Y)$. This kind of translation coverings are often called \textit{balanced} translation coverings.
It is a result of M\"oller in \cite{Moe06}, that every translation surface is the (balanced) covering surface of a primitive base surface and that the primitive surface is unique, if its genus is greater than one.\\

The \textit{affine group} of a translation surface $(X, \omega)$ is the set of affine, orientation preserving diffeomorphisms on $X$, i.e.\ maps that can locally be written as $z \mapsto A z + b$ with $A \in \SL$ and $b \in \R^2$. The translation vector $b$ depends on the local coordinates, whereas the derivative $A$ is globally defined. The derivatives of the affine maps on $X$ (i.e.\ of the elements in the affine group), form the \textit{Veech group} $\Gamma(X)$ of $X$. An important connection between the Veech groups of a primitive base surface and the covering surface in a translation covering is stated in the following lemma.

\begin{lemma}
\label{Lemma:prim}
 Let $p:(Y, \nu) \to (X, \omega)$ be a (finite, balanced) translation covering with primitive base surface $(X, \omega)$ and genus $g(X) > 1$, then
 $\Gamma(Y)$ is a subgroup of $\Gamma(X)$.
\end{lemma}
The Lemma is an immediate consequence of the result of M\"oller cited above (see \cite{McM06}).
In Chapter \ref{sec:inf}, the statement of Lemma~\ref{Lemma:prim} will be proved separately as Lemma~\ref{Lemma:inf} for finite and infinite coverings of the base surfaces $X_n$, which are introduced in Chapter \ref{sec:baseSurface}.\\

We will often use the \textit{cylinder decomposition} of a translation surface $X$. A \textit{cylinder} is a maximal connected set of homotopic simple closed geodesics. The \textit{inverse modulus} of a cylinder is the ratio of its length to its height. Whenever the genus of $X$ is greater than one, every cylinder is bounded by geodesic intervals joining singular points. Such a geodesic is called \textit{saddle connection} if no singular point lies in the interior. The \textit{Veech alternative} states, that if the Veech group of a translation surface is a lattice, then the geodesic flow in each direction is either periodic or uniquely ergodic. 
This important result of Veech implies in particular, that if $X$ has a lattice Veech group and if there exists a closed geodesic in direction $\theta$ on $X$ then the surface decomposes into cylinders in direction $\theta$.\\

Being a compact Riemann surface, $X$ defines a point in the moduli space $M_{g,s}$ of Riemann surfaces of genus $g$ with $s = |\Sigma(X)|$ punctures.
Every matrix $A \in \SL$ can be used to change the translation structure $\omega$ on $X$, by composing each chart with the affine map $z \mapsto A \cdot z$. We call the new translation surface $X_A$. The identity map $\id_A :X \to X_A$ is an orientation preserving homeomorphism, so $(X_A, \id_A)$ is a point in the Teichm\"uller space $T_{g,s}$. The surfaces $X_A$ and $X_B$ define the same point in $T_{g,s}$ iff $A \cdot B^{-1} \in \SO$ thus defining a map $i : \SL/\SO \cong \HH \hookrightarrow T_{g,s}$. 
The Veech group $\Gamma(X)$ lies in the stabiliser in $\SL$ of $X_I$ as Riemann surface.
Hence we obtain a map from $\HH / \Gamma(X)$ to $M_{g,s}$. 
By the Theorem of Smillie the image of this map, which is birational to $\HH / \Gamma(X)$, is an algebraic curve in the moduli space iff $\Gamma(X)$ is a lattice in $\SL$. In this case, this curve is called a \textit{Teichm\"uller curve}.

\section{The base surfaces.}
\label{sec:baseSurface}

\subsection{The odd regular double-n-gon.}
\label{sec:nGon}
The translation surface described in this section, the odd regular double-n-gon $X_n$, was already considered by Veech himself in \cite{Vee89}. Other references concerning this translation surface are \cite{HS01} Chapter~1.7 and \cite{Vor96} Chapter~4. The surface arises from a triangle with angles $\pi/n$, $\pi/n$ and $(n-2)\pi/n$ by the unfolding construction described in \cite{ZK75}.

In the following, let $n$ be an odd number, greater or equal to $5$. The translation surface $X_n$ can be glued of two regular n-gons $P_n$ and $Q_n$. Rotate the $n$-gons until each of them has a horizontal side and $P_n$ lies above its horizontal side, whereas $Q_n$ lies below it. Gluing parallel sides leads to the desired translation surface. To fix lengths, let the circumscribed circle of $P_n$ (respectively of $Q_n$) have radius $1$. Then the edges of the $n$-gons have lengths $2 \cdot \sin(\frac{\pi}{n})$.
\begin{figure}[htb]
 \centering
 \includegraphics{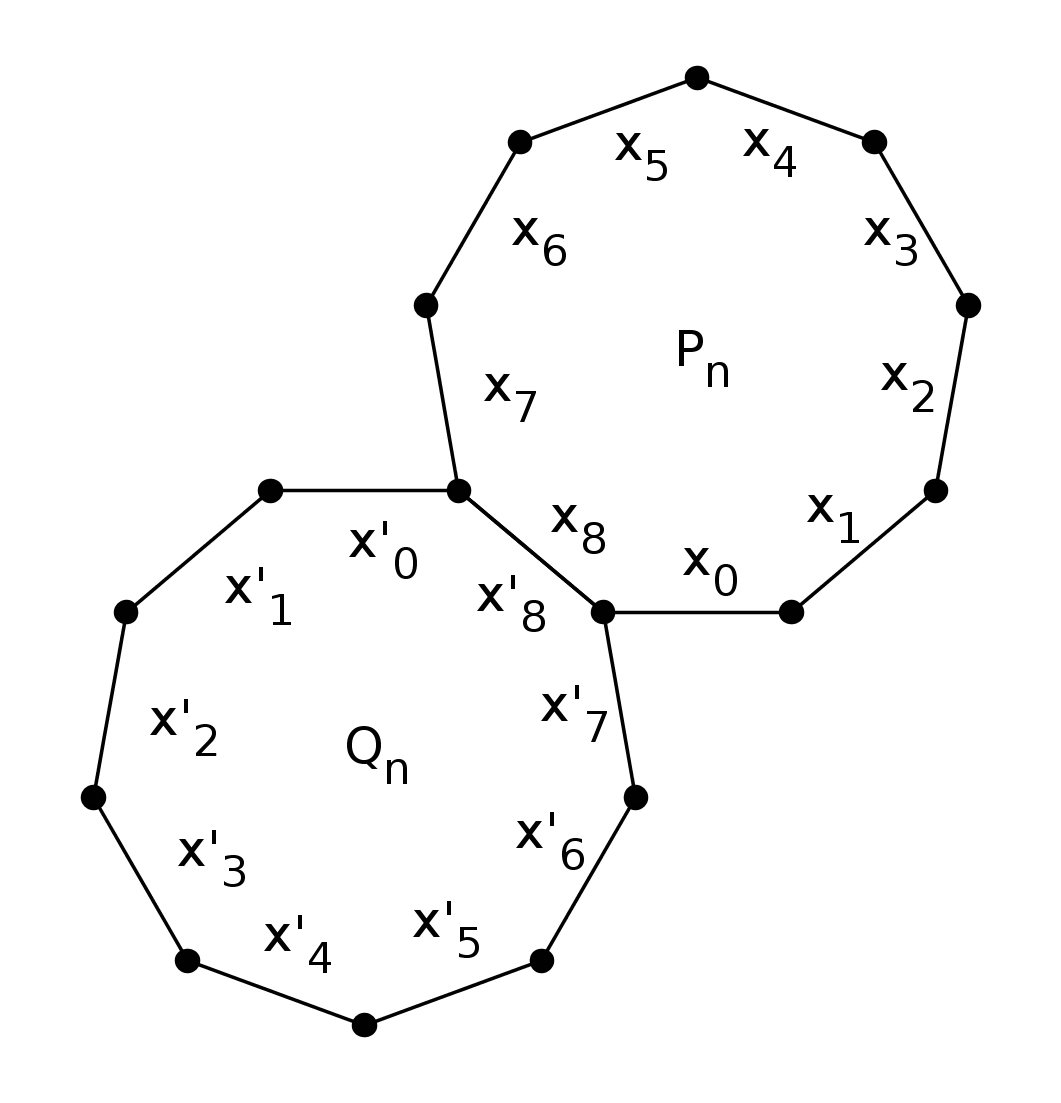}
 \caption{Labelling of $X_9$}
 \label{fig:vertexNr}
\end{figure}
We label the edges of the two polygons in the following way:
Label the edges of $P_n$ and $Q_n$ counter-clockwise with $x_0, x_1, \dots, x_{n-1}$ and $x'_0, x'_1, \dots,$ $x'_{n-1}$, starting with the horizontal side (see Figure \ref{fig:vertexNr}).
Whenever needed, we initially glue the edges $x_{n-1}$ and $x'_{n-1}$ to obtain only one polygon that defines $X_n$.

\begin{remark}
\label{rem:genus}
 Following the identification of vertices while gluing the edges $x_i$ and $x_i'$ one can easily see that all the vertices are identified. $X_n$ thereby has exactly one singularity with conical angle $2 n \, \frac{n-2}{n} \pi = (n-2) \, 2 \pi$. The Euler characteristic is 
  $\chi = 1 - n + 2 = 3-n$
 and the genus of $X_n$ is $g(X_n) = \frac{n-1}{2}$. 
\end{remark}

According to \cite{Vee89} Theorem 5.8, the Veech group  of $X_n$ is generated by the matrices 
$$R_n := \begin{pmatrix} \cos{\frac{\pi}{n}} & -\sin{\frac{\pi}{n}} \\ \sin{\frac{\pi}{n}} & \cos{\frac{\pi}{n}} \end{pmatrix} 
\quad \textrm{and} \quad 
T_n := \begin{pmatrix} 1 & \lambda_n \\ 0 & 1 \end{pmatrix} 
\quad \textrm{where} \quad
\lambda_n = 2 \cot{\frac{\pi}{n}} \,.$$

Let $\varphi_R$ denote the orientation preserving affine diffeomorphism of $X_n$ with derivative $R_n$. The map  $\varphi_R$ rotates $X_n$ around the centre $m_P$ of $P_n$ and further translates $X_n$ along the vector from $m_P$ to the centre $m_Q$ of $Q_n$. 

To understand the affine map on $X_n$ with derivative $T_n$, we use a cylinder decomposition of $X_n$. 
The horizontal saddle connections in $X_n$ decompose the translation surface into $\frac{n-1}{2}$ cylinders.
Figure \ref{fig:CylinderSizes} indicates how to compute the heights and lengths of these cylinders: First we rotate the double-$n$-gon counterclockwise by $90$ degrees. If we choose the origin of a coordinate system in the centre of $P_n$, then the vertices of $P_n$ lie in $( \, \cos (j \pi \, \frac{2}{ n}), \, \sin (j \pi \, \frac{2}{n}) \, )$, $j = 0, \dots, n-1$. This is all the information we need.

\begin{figure}[htb]
 \centering
 \includegraphics{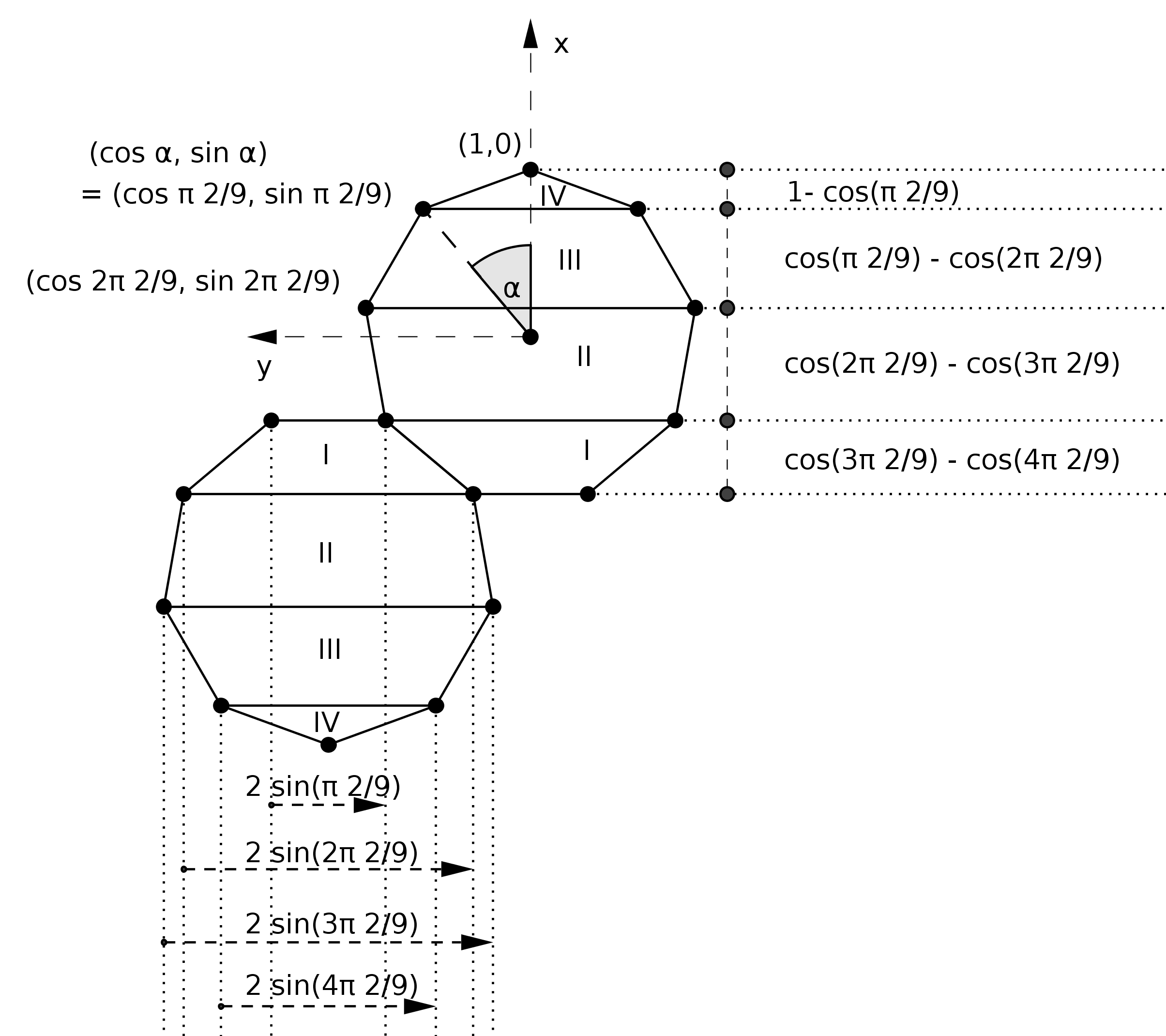}
 \caption{Horizontal cylinder decomposition of $X_9$}
 \label{fig:CylinderSizes}
\end{figure}

We name the cylinder containing the gluing along the edges $x_i$ and $x_{n-i}$ with $i$ (see Figure~\ref{fig:CylinderSizes}). If $h_i$ denotes the height and $l_i$ the length of the $i$th cylinder and $j= \frac{n+1}{2} - i$, then

\begin{equation}
\label{eq:height}
\begin{array}{rcl}
  h_i 	&=& \cos{\frac{2 \pi}{n} (j-1)} - \cos{ \frac{2 \pi}{n} j} \\ \noalign{\medskip}
	&=& -2 \sin{\frac{\pi (2 j -1)}{n}} \sin{\frac{-\pi}{n}} = 2 \sin{\frac{\pi (2 j -1)}{n}} \sin{\frac{\pi}{n}}
\end{array}
\end{equation}
 and
\begin{equation}
\label{eq:width}
\begin{array}{rcl}
 l_i 	&=& 2 (\sin{\frac{2 \pi}{n} j} + \sin{ \frac{2 \pi}{n} (j-1)}) \\\noalign{\medskip}
	&=& 4 \sin{\frac{\pi (2 j -1)}{n}} \cos{\frac{\pi}{n}}
\end{array}
\end{equation}
for $i \in \{1, \dots, \frac{n-1}{2} \}$.
Note that we used the following consequence of the addition and subtraction theorems for $\sin$ and $\cos$:
$$\begin{array}{rcl}
   \cos x - \cos y &=& - 2 \sin(\frac{x+y}{2}) \sin(\frac{x-y}{2})\\
   \sin x + \sin y &=& 2 \sin(\frac{x+y}{2}) \cos(\frac{x-y}{2})
  \end{array}
$$ 

The inverse modulus in all horizontal cylinders of $X_n$ is $\lambda_n = 2 \cot{\frac{\pi}{n}}$. Thus we can construct an orientation preserving affine diffeomorphism $\varphi_T$ of $X_n$ with derivative $T_n$ (see e.g. \cite{Vor96} Chapter 3.2).
The map $\varphi_T$ fixes all horizontal saddle connections pointwise and twists every cylinder once.

\begin{remark}
 The Veech group $\Gamma(X_n)$ is the Hecke triangle group with signature $(2,n,\infty)$, thus it is a lattice.
\end{remark}

It is well known, that the translation surface $X_n$ is primitive. We give a short proof nevertheless.
\begin{lemma}
\label{Lemma:primitiveOdd}
 The translation surface $X_n$ with odd $n\geq5$ is primitive.
\end{lemma}
\begin{proof}
 Suppose $p: X_n \to Y$ is a translation covering of degree $d > 1$. The surface $X_n$ has one singularity, say $x$, so $p$ is ramified at most over $p(x)$. We have to distinguish two cases. 
 
 If $y = p(x)$ is a removable singularity in the flat structure, than $Y$ is w.l.o.g. the once punctured torus.
  According to \cite{GJ00} a translation surface is the covering surface of the once punctured torus if and only if its Veech group is arithmetic, i.e.\ commensurable to $\textrm{SL}_2(\Z)$. But the Veech group of $\Gamma(X_n)$ is not arithmetic.
 
 In the second case, $y$ is a non removable singularity in the flat structure. Then $p^{-1} (y) = \{x\}$.
 It follows that the preimage of a saddle connection on Y is the union of $d$ saddle connections on $X_n$.
 The map $p$ is locally a translation, so theses $d$ saddle connections all have the same length and direction. Since there is only one shortest saddle connection on $X_n$ in horizontal direction, this leads to a contradiction to $d>1$.
 
 We conclude that every translation covering $p:X_n \to Y$ has degree $1$, i.e. the genus of $Y$ equals the genus of $X$.
 \end{proof}

In particular $X_n$ has no translations, with the consequence that there is exactly one affine map with derivative $A$ for each $A$ in the Veech group $\Gamma(X_n)$.\\

According to Remark \ref{rem:genus}, the translation surface $X_n$ has one singularity and genus $\frac{n-1}{2}$. Consequently the fundamental group $\pi_1(X_n \setminus \Sigma)$ is free of rank $n-1$.
We use the centre of the side $x_{n-1}$ as base point. A basis of the fundamental group of the double-$n$-gon can be chosen in such a way, that the $i$th generator ($i \in \{0, \dots, n-2\}$) crosses the edge $x_i$ once from $P_n$ to $Q_n$ (and respectively the edge $x_{n-1}$ once from $Q_n$ to $P_n$). The $i$th generator of the fundamental group will be called $x_i$, like the label of the edge it crosses (see Figure \ref{fig:fundamentalGroup}). 
Now an arbitrary element of the fundamental group $\pi_1(X_n \setminus \Sigma)$ can be factorised in this basis by recording the names of the crossed edges and the directions of the crossings. The edge $x_{n-1}$ corresponds to the identity element.
\begin{figure}[htb]
 \centering
 \includegraphics{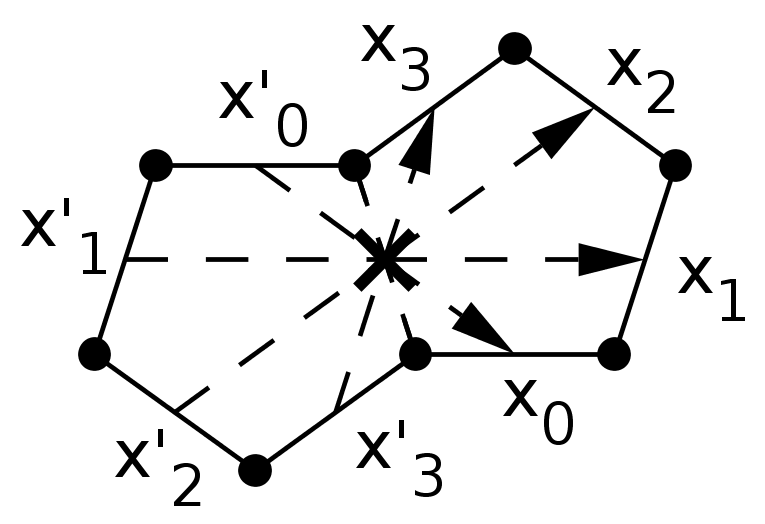}
 \caption{Fundamental group of the double-$5$-gon}
 \label{fig:fundamentalGroup}
\end{figure}

\subsection{The even regular n-gon.}

Now let $n$ be an even number and $n \geq 8$. The translation surface $X_n$ arises from a regular $n$-gon $P$, with opposite sides glued together. As indicated in the introduction, the regular double $n$-gon, considered by Veech in \cite{Vee89}, is a degree-$2$-covering of this surface. It is well known, that the regular $n$-gon is a primitive translation surface.
To fix the size and orientation of the polygon, let the $j$th vertex of $P$ lie in $(\cos(j \frac{2 \pi}{n}), \sin(j \frac{2 \pi}{n}))$, $j \in \{0, \dots, n-1\}$. We label the edges $P$ counter-clockwise with $x_0, \dots, x_{n/2-1}, x'_0, \dots, x'_{n/2-1}$, starting with the edge between vertex $0$ and vertex $1$ (see Figure \ref{fig:Xn_even_labels}).
\begin{figure}[htbp]
\centering
\subfigure[$X_8$]{
  \centering
  \includegraphics{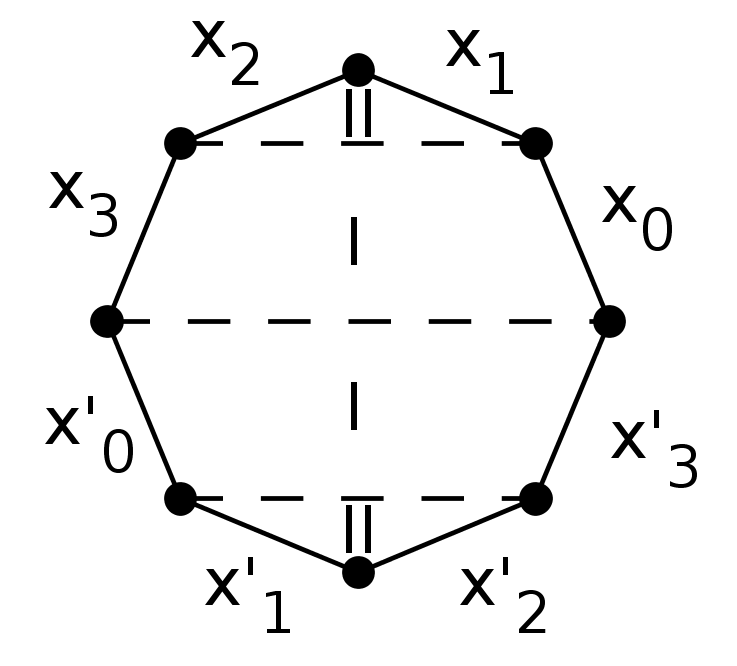}
  \label{X8}
}
\subfigure[$X_{10}$]{
  \centering
  \includegraphics{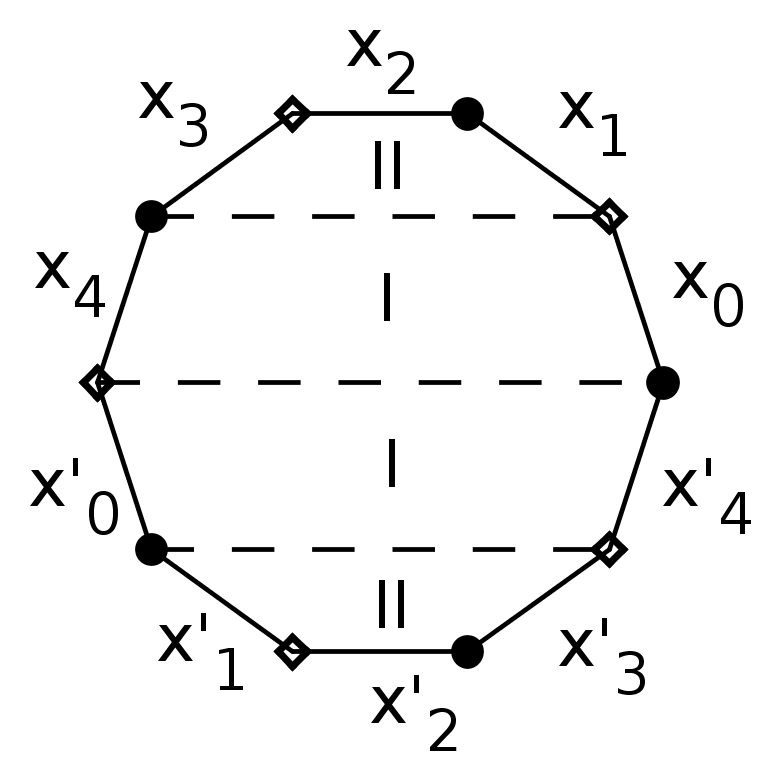}
  \label{X10}
}
 \caption{Labelling of $X_n$ for even $n$}
 \label{fig:Xn_even_labels}
\end{figure}

\begin{remark}
\label{rem:genus_even}
 The surface $X_n$ has one singularity, if $n/2$ is even
 and two singularities if $n/2$ is odd.
 Using the Euler characteristic it follows, that $X_n$ has genus 
 $$g(X_n) = \left\{ \begin{array}{ll}
                n/4 & \textrm{, if } n \equiv 0 \mmod 4\\
		(n-2)/4 & \textrm{, if } n \equiv 2 \mmod 4
               \end{array}
\right. \, .$$
\end{remark}

Lemma J in \cite{HS01} states, that the Veech group of $X_n$ equals the Veech group of Veech's double-$n$-gons for even $n \geq 8$. So according to \cite{Vee89} Theorem~5.8, $\Gamma(X_n) = \langle  T_n, {R_n}^2, R_n T_n {R_n}^{-1} \rangle$ where 
$$R_n := \begin{pmatrix} \cos{\frac{\pi}{n}} & -\sin{\frac{\pi}{n}} \\ \sin{\frac{\pi}{n}} & \cos{\frac{\pi}{n}} \end{pmatrix} 
\quad, \quad 
T_n := \begin{pmatrix} 1 & \lambda_n \\ 0 & 1 \end{pmatrix} 
\quad \textrm{and} \quad
\lambda_n = 2 \cot{\frac{\pi}{n}} \,$$
 are defined as in the odd case. Since $R_n T_n {R_n}^{-1} = {R_n}^{n+2} \, {T_n}^{-1}$ (see relations on page~\pageref{Lemma:VeechGroupGen}), only the first two generators are needed.\\

Again we need the cylinder decomposition of $X_n$ in horizontal direction to understand the affine map on $X_n$ with derivative $T_n$.
 The surface $X_n$ decomposes into $n/4$ or $(n-2)/4$ cylinders, depending on $n \equiv 0 \mmod 4$ or $n \equiv 2 \mmod 4$. If we label the cylinder containing the gluing along the edges $i$ and $(n/2-1) - i$ with $i+1$ (see Figure \ref{fig:Xn_even_labels}), we get 
\begin{equation}
\label{eq:widthAndHeightEven}
 h_i = 2 \cos{\frac{(2 i -1) \pi }{n}} \sin{\frac{\pi}{n}}
 \quad \textrm{and} \quad
  l_i = 4 \cos{\frac{ (2 i - 1) \pi}{n}} \cos{\frac{\pi}{n}}
\end{equation}
for $i \in \{1, \dots, \frac{n}{4}\}$  or $i \in \{ 1, \dots, \frac{n-2}{4}\}$, respectively.
It follows that the inverse modulus is $2 \cot \frac{\pi}{n} = \lambda_n$ in all horizontal cylinders.

\begin{remark}
 For even $n$, the rotation $R_n$ is not contained in the Veech group $\Gamma(X_n)$. In this case $\Gamma(X_n)$ is a triangle group with signature $(\frac{n}{2}, \infty, \infty)$ and an index $2$ subgroup of the Hecke triangle group $\langle T_n, R_n \rangle$ with signature $(2, n, \infty)$.
\end{remark}

The translation surface $X_n$ has one singularity and genus $n/4$ if  $n \equiv 0 \mmod 4$ or two singularities and genus $(n-2)/4$ if  $n \equiv 2 \mmod 4$. In both cases, the fundamental group $\pi_1(X_n \setminus \Sigma)$ is free of rank $n/2$. 
We use the centre as base point and choose closed paths that cross exactly one edge $x_i$ one time as basis of the fundamental group (see Figure \ref{fig:fundamentalGroupEven}). Now an arbitrary element of the fundamental group $\pi_1(X_n \setminus \Sigma)$ can be factorised in this basis by recording the names of the crossed edges and the directions of the crossings.
\begin{figure}[htb]
 \centering
 \includegraphics{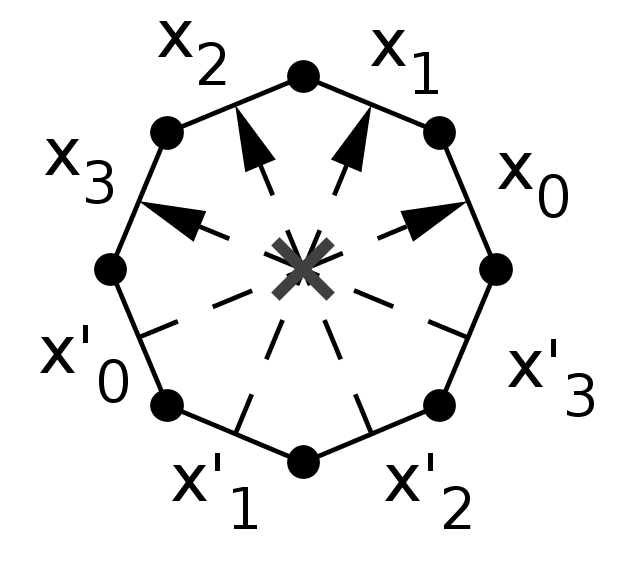}
 \caption{Fundamental group of the regular $8$-gon}
 \label{fig:fundamentalGroupEven}
\end{figure}

\section{A series of n-gon coverings.}

Using a generalised algorithm of the one presented in \cite{Sch04}, it is possible to calculate the Veech group of an arbitrary finite covering of a double-$n$-gon or $n$-gon (see \cite{Fr08}). With the help of such computations, a conjecture about an interesting family of translation coverings
$$p_{n,d}: Y_{n,d} \to X_n$$ 
arose. 
In the following we define the family $p_{n,d}$ and prove that 
the Veech group of the covering surface $Y_{n,d}$ is
$$\Gamma(Y_{n,d}) = 
\langle -I, T_n, R_n \, {T_n}^2 \, {R_n}^{-1}, \dots, {R_n}^{\frac{n-1}{2}} \ {T_n}^2 \, {R_n}^{-\frac{n-1}{2}} \rangle$$
for odd $n \geq 5$ and
 $$\begin{array}{rl} 
 \Gamma(Y_{n,d}) = &\langle -I, T_n, {R_n}^2 \, {T_n}^2 \, {R_n}^{-2}, \dots, {R_n}^{n-2} \, {T_n}^2 \, {R_n}^{-(n-2)}, \, ({T_n}^{-1}{R_n}^2)^2, \\ \noalign{\smallskip}
     & {R_n}^2 \,({T_n}^{-1}{R_n}^2)^2 \, {R_n}^{-2}, \dots, {R_n}^{n-2} \, ({T_n}^{-1}{R_n}^2)^2 \, {R_n}^{-(n-2)} \rangle
   \end{array} $$
for even $n \geq 8$
with $R_n$ and $T_n$ as in Chapter \ref{sec:nGon}.
In particular, it is independent of the covering degree $d$.
Our proof uses geometric arguments and not the methods, used to develop the algorithm mentioned above.

The series of coverings is somehow similar to the stair-origamis in \cite{Sch06} and to the $Z_{2,0}^k$ series in \cite{SchHu09}. It likewise uses two slits in the base surface to construct the covering surface and has a Veech group that is independent of the covering degree.\\

We will define the coverings by their monodromy. We recall the definition of the monodromy; for more details 
see \cite{Mir95} Chapter III.4. 
Let $p_{n,d}: Y_{n,d} \to X_n$ be a covering of degree $d$ and $\Sigma = \Sigma(X_n)$.
Choose a base point $x$ in $X_n \setminus \Sigma$ and call its preimages in $Y_{n,d}$ $0, \dots, d-1$. Every closed path in $X_n \setminus \Sigma$ at $x$ can be lifted to a path in $Y_{n,d}$ with starting point in $\{0, \dots, d-1\}$. The end point of the lifted path is again contained in $\{0, \dots, d-1\}$ and the lifts define a permutation of the points $\{0, \dots, d-1\}$ in $Y_{n,d}$. 
The corresponding map $$m: \pi_1(X_n \setminus \Sigma, x) \rightarrow S_d$$
has the following property
$$m([w_1] \cdot [w_2]) = m([w_2]) \circ m([w_1]) \quad \forall \, [w_1], [w_2] \in \pi_1(X_n \setminus \Sigma,x)$$ and is called  the \textit{monodromy} of the covering. Its image in $S_d$ is a transitive permutation group.
On the other hand, every such map $m: \pi_1(X_n \setminus \Sigma, x) \rightarrow S_d$ with transitive image defines a degree $d$ covering of $X_n$.\\

To define the map $m$ it is sufficient to define the images of the generators $x_0, \dots, x_{n-2}$ (or $x_0, \dots, x_{\frac{n}{2}-1}$ respectively) of $\pi_1(X_n \setminus \Sigma, x)$ in $S_d$. The generator $x_i$ crosses the edge $x_i$ and no other edge, so that the permutation $\sigma_i = m(x_i)$ indicates directly how the $d$ copies of $X_n$ are glued along the edges $x_i$ and $x_i'$ to obtain the covering surface $Y_{n,d}$. The edge $x_i$ in copy $j$ is glued to the edge $x_i'$ in copy $\sigma_i(j)$.

\subsection{Definition of the coverings.}
\label{sec:coveringDefs}

In the following, we define for each $d \geq 2$ a covering $p: Y_{n,d} \to X_n$ by giving its monodromy $m_{n,d}$.
For the definition, we need the two permutations
$$\sigma_{d,1} =
\left\{ \begin{array}{ll}
 (0 \; 1) \, (2 \; 3) \, \cdots \, (d-2 \;\; d-1) &, d \textrm{ even}\\
 (0 \; 1) \, (2 \; 3) \, \cdots \, (d-3 \;\; d-2) &, d \textrm{ odd}
\end{array}\right.$$
and 
$$\sigma_{d,2} =
\left\{ \begin{array}{ll}
 (1 \; 2) \, (3 \; 4) \, \cdots \, (d-3 \;\; d-2) \, (d-1 \;\; 0) &, d \textrm{ even}\\
 (1 \; 2) \, (3 \; 4) \, \cdots \, (d-2 \;\; d-1) &, d \textrm{ odd}
\end{array}\right. \, .$$
Further, let
$$k_1 = \left\{\begin{array}{ll}
  \frac{n-1}{2} & \textrm{, } n \textrm{ odd}\\
  \frac{n}{4} -1 & \textrm{, } n \equiv 0 \mmod 4\\
  \frac{n-2}{4} -1 & \textrm{, } n \equiv 2 \mmod 4
  \end{array}\right. 
  \quad \textrm{ and } \quad 
  k_2 = \left\{\begin{array}{ll}
  \frac{n+1}{2} & \textrm{, } n \textrm{ odd}\\
  \frac{n}{4} & \textrm{, } n \equiv 0 \mmod 4\\
  \frac{n-2}{4} + 1 & \textrm{, } n \equiv 2 \mmod 4
  \end{array}\right. \, .$$
For $d \geq 2$ define the monodromy $m_{n,d}$ by:
$$ m_{n,d}: \left\{ \begin{array}{lcll}
               	\pi_1(X_n \setminus \Sigma) & \longrightarrow & S_d \\
		x_i & \mapsto & \id &,i \notin \{k_1, k_2\}\\
		x_{k_1} & \mapsto & \sigma_{d,1} \\
		x_{k_2} & \mapsto & \sigma_{d,2}
              \end{array}
\right. $$
Let $Y_{n,d}$ denote the covering surface of the translation covering $p:Y_{n,d} \rightarrow X_n$, defined by the monodromy $m_{n,d}: F_{n-1} \rightarrow S_d$ or $m_{n,d}: F_{\frac{n}{2}} \rightarrow S_d$, respectively.
Figure \ref{fig:CoveringExampels} shows the degree $2$, $3$ and $4$ coverings of the double-$5$-gon, as well as the degree $2$ and $3$ covering of the $8$-gon.
The non labelled edges are glued to the parallel edge in the same $X_n$ copy, labelled edges are glued to their labelled correspondent.

\begin{figure}[htbp]
\centering
\subfigure[$X_5$ degree $2$]{
\centering
 \includegraphics{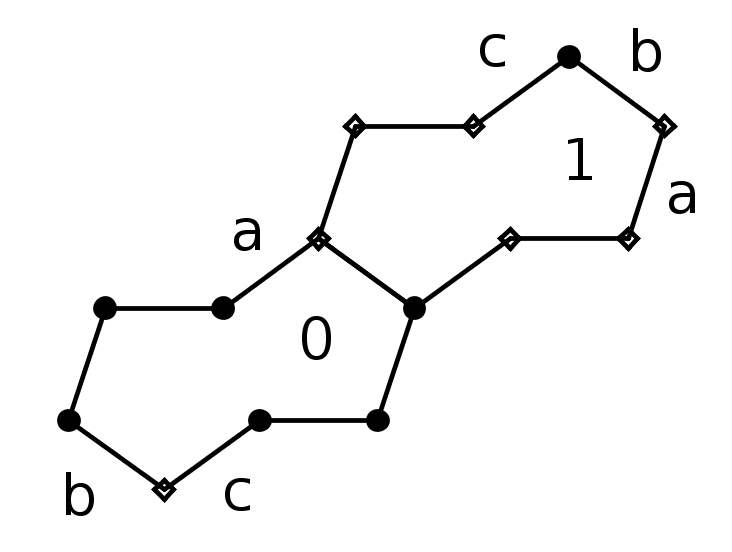}
}
\subfigure[$X_8$ degree $2$]{
  \centering
 \includegraphics{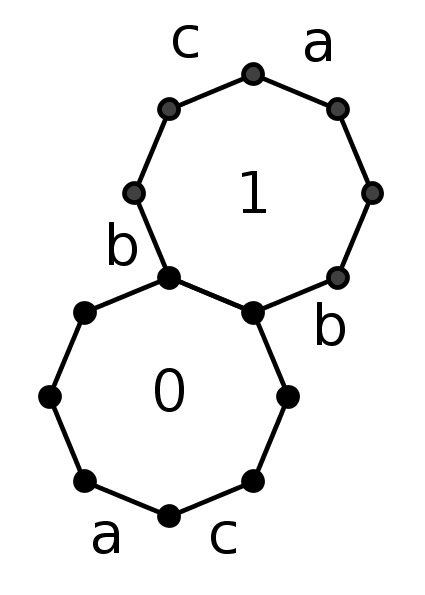}
}
\\
\subfigure[$X_5$ degree $3$]{
\centering
 \includegraphics{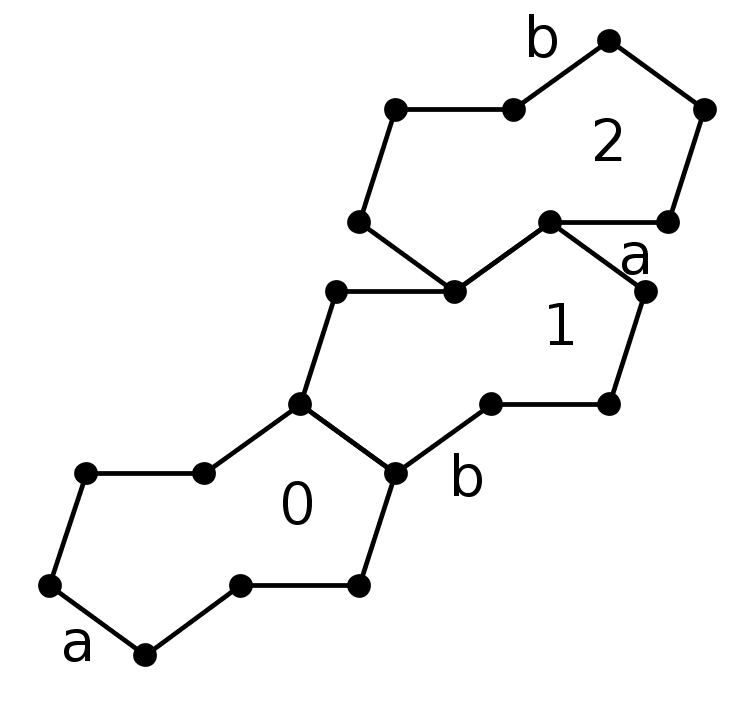}
}
 \subfigure[$X_8$ degree $3$]{
  \centering
  \includegraphics{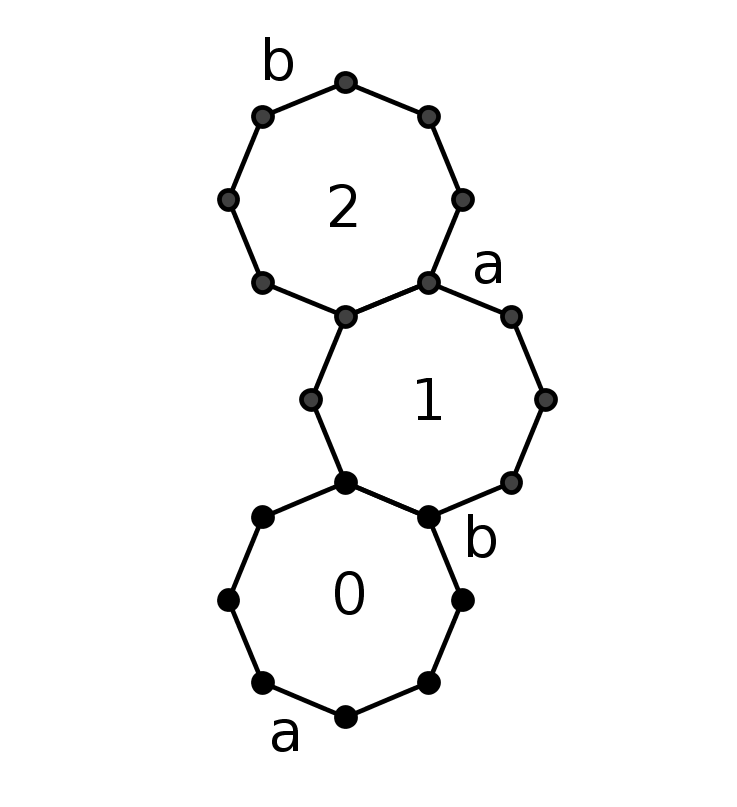}
}
\subfigure[$X_5$ degree $4$]{
\centering
 \includegraphics{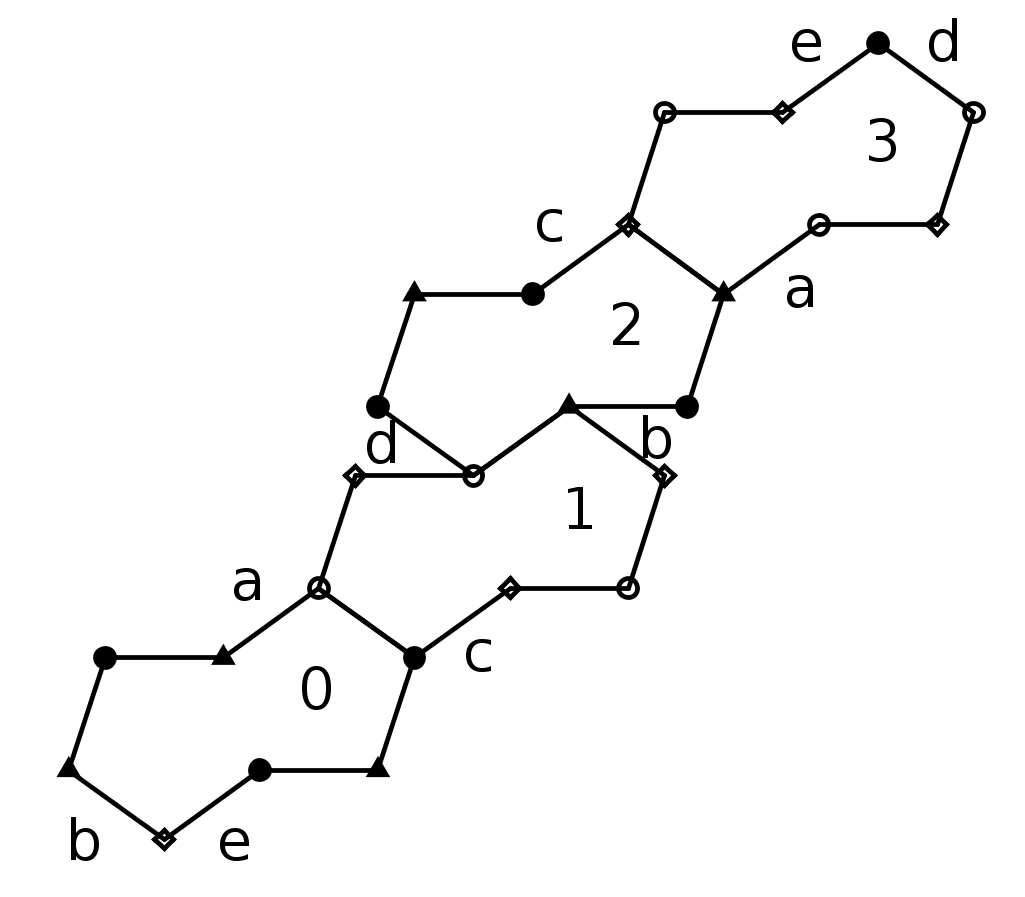}
}
 \caption{Translation coverings $Y_{n,d}$ of $X_n$}
 \label{fig:CoveringExampels}
\end{figure}
Figure \ref{fig:OpMdeven} and Figure \ref{fig:OpMdodd} show the action of the generators of $\pi_1(X_n \setminus \Sigma)$ via $m_{n,d}$ on
the set $\{0, \dots, d-1\}$. 
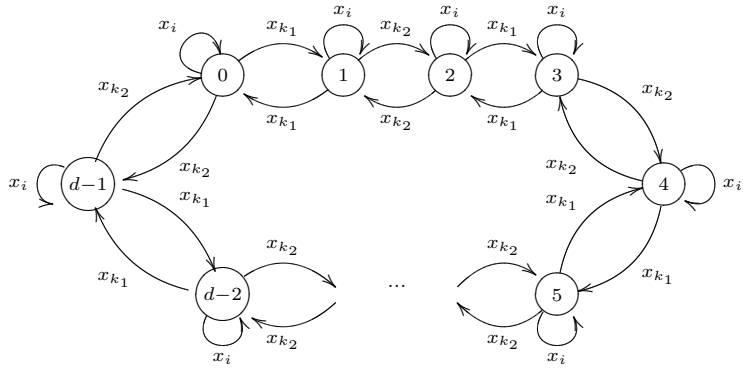
\begin{figure}[htbp]
\centering
\def\objectstyle{\scriptstyle}
\def\labeltstyle{\scriptstyle}
\xymatrix{
 &&
 &
 *++[o][F-]{0} \ar@(l,u)[]^{x_i} 	\ar@/^1pc/[r]^{x_{k_1}} 	\ar@/^1pc/[dl]^{x_{k_2}} 	&
 *++[o][F-]{1} \ar@(ul, ur)[]^{x_i} 	\ar@/^1pc/[l]^{x_{k_1}} 	\ar@/^1pc/[r]^{x_{k_2}}		&
 *++[o][F-]{2} \ar@(ul, ur)[]^{x_i} 	\ar@/^1pc/[r]^{x_{k_1}} 	\ar@/^1pc/[l]^{x_{k_2}} 	&
 *++[o][F-]{3} \ar@(ul, ur)[]^{x_i} 	\ar@/^1pc/[l]^{x_{k_1}} 	\ar@/^1pc/[rd]^{x_{k_2}} 	&
\\
 &&
 *++[o][F-]{d-1}\ar@(lu, ld)[]_{x_i} 	\ar@/^1pc/[dr]^{x_{k_1}}	\ar@/^1pc/[ur]^{x_{k_2}}  & & & & & 
 *++[o][F-]{4}	\ar@(ru, rd)[]^{x_i}	\ar@/^1pc/[dl]^{x_{k_1}}	\ar@/^1pc/[ul]^{x_{k_2}}
\\
 && 
 & *++[o][F-]{d-2}\ar@(dl, dr)[]_{x_i} 	\ar@/^1pc/[ul]^{x_{k_1}}	\ar@/^1pc/[r]^{x_{k_2}} &
 \ar@/^1pc/[l]^{x_{k_2}} \ar@{}[r]^{\ldots}	& 
 \ar@/^1pc/[r]^{x_{k_2}}	&
 *++[o][F-]{5}	\ar@(dl, dr)[]_{x_i}	\ar@/^1pc/[ur]^{x_{k_1}}	\ar@/^1pc/[l]^{x_{k_2}}
 }
\caption{Monodromy action for even $d$}
\label{fig:OpMdeven}
\end{figure}
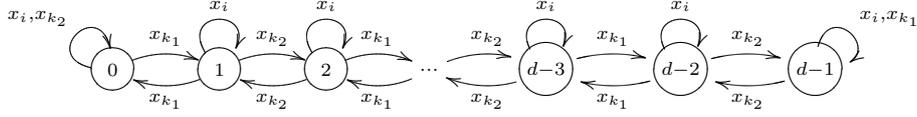
\begin{figure}[htbp]
\centering
\def\objectstyle{\scriptstyle}
\def\labeltstyle{\scriptstyle}
\xymatrix{
 *++[o][F-]{0} \ar@(l,u)[]^{x_i, x_{k_2}} 	\ar@/^/[r]^{x_{k_1}} 			&
 *++[o][F-]{1} \ar@(ul, ur)[]^{x_i} 	\ar@/^/[l]^{x_{k_1}} 	\ar@/^/[r]^{x_{k_2}}	&
 *++[o][F-]{2} \ar@(ul, ur)[]^{x_i} 	\ar@/^/[r]^{x_{k_1}} 	\ar@/^/[l]^{x_{k_2}} 	&
 \ldots \ar@/^/[l]^{x_{k_1}}  	\ar@/^/[r]^{x_{k_2}}					& 
 *++[o][F-]{d-3} \ar@(ul, ur)[]^{x_i} 	\ar@/^/[r]^{x_{k_1}} 	\ar@/^/[l]^{x_{k_2}} 	&
 *++[o][F-]{d-2} \ar@(ul, ur)[]^{x_i} 	\ar@/^/[l]^{x_{k_1}} 	\ar@/^/[r]^{x_{k_2}} 	&
 *++[o][F-]{d-1} \ar@(u,r)[]^{x_i, x_{k_1}} 	\ar@/^/[l]^{x_{k_2}} 
}
\caption{Monodromy action for odd $d$}
\label{fig:OpMdodd}
\end{figure}

\begin{remark}
  We excluded $n = 4$ and $n=6$ in our considerations, because the genus of $X_4$ and $X_6$ is one. In these cases the singularities of the base surface $X_n$ are removable (i.e. have angle $2 \pi$). Hence it is an additional assumption, that affine maps on the surface map singularities to singularities and that covering maps $p: Y \to X_n$ satisfy $p^{-1}(\Sigma(X_n)) = \Sigma(Y)$.
  If $n=4$ the covering surface $Y_{4,d}$ for even degree $d = 2l$ is the lattice surface $\Gamma_{2,0}^l$ defined in \cite{SchHu09} and the limit of our series is their lattice surface $\Gamma_{2,0}^\infty$.
  For $n=6$, considerations similar to Proposition 2.6.\! in \cite{Sch04} show that every affine map $f$ on $Y$ (respecting $\Sigma(Y) = p^{-1}(\Sigma(X_6))$ ) descends via a translation covering $p:Y \to X_6$ to an affine map $\tilde{f}$ on the twice punctured torus $X_6$ with $p \circ f = \tilde{f} \circ p$, i.e. $\Gamma(Y) \subseteq \Gamma(X_6) = \langle T_6, {R_6}^2 \rangle$. By considering some extra cases in the proof, it can be seen that Theorem 1 b)
  also holds for $n=6$.
\end{remark}

\subsection{The Veech group of the odd series.}
\label{sec:VeechUeberlagerungOdd}
The goal of this section is the proof of our main theorem for odd $n\geq5$.

\begin{ThmOdd}
For odd $n\geq 5$, the Veech group of $Y_{n,d}$ is
 $$\Gamma_n := \Gamma(Y_{n,d}) = \langle -I, T_n, R_n \, {T_n}^2 \, {R_n}^{-1}, \dots, {R_n}^{\frac{n-1}{2}} \ {T_n}^2 \, {R_n}^{-\frac{n-1}{2}} \rangle\;.$$
The matrices $R_n$ and $T_n$ are defined as in Chapter \ref{sec:nGon} 
as
$$R_n = \begin{pmatrix} \cos{\frac{\pi}{n}} & -\sin{\frac{\pi}{n}} \\ \sin{\frac{\pi}{n}} & \cos{\frac{\pi}{n}} \end{pmatrix} \; \textrm{and } \; T_n = \begin{pmatrix} 1 & \lambda_n \\ 0 & 1 \end{pmatrix} \; \textrm{where } \lambda_n= 2 \cot{\frac{\pi}{n}}$$
and $I \in \GL$ is the identity matrix.
\end{ThmOdd}
In particular, $\Gamma_n$ is a subgroup of $\Gamma(X_n)$ of index $n$ (see Lemma \ref{Lemma:VeechGroupGen}) and independent of the covering degree $d$. At the end of this section, we deduce the basic properties of the Teichm\"uller curve defined by the translation surface $Y_{n,d}$ in Corollary \ref{cor:Curve}.
We start the proof of the theorem with the following lemma.
\begin{lemma}
 \label{Lemma:VeechGroupGen} 
 The group $G := \langle -I, T_n, R_n \, {T_n}^2 \, {R_n}^{-1}, \dots, {R_n}^{\frac{n-1}{2}} \ {T_n}^2 \, {R_n}^{-\frac{n-1}{2}} \rangle$ is of index $n$ in $\Gamma(X_n) = \langle R_n, T_n \rangle$. The set of cosets is
 $$G \bs \Gamma(X_n) = \{G \, I , G \, R_n, \dots, G \, {R_n}^{n-1} \}\, ,$$
 where $I \in \GL$ denotes the identity matrix.
\end{lemma}

\begin{proof}
 Write $\Gamma(X_n) = F/N$, where $F$ is the free group on the generators $R$ and $T$ and $N$ is the normal subgroup corresponding to the relations of $R$ and $T$ in $\Gamma(X_n)$.
  The image of $\Gamma(X_n)$ in $\PSL$ is a triangle group with signature $(\infty,n,2)$, so a possible set of defining relations in $\PSL$ is $\{ R^n, (T^{-1} \, R)^2 \}$. An easy calculation shows, that $R^n = -I = (T^{-1} \, R)^2$. We deduce (see e.g. \cite{Joh97} Chapter 10), that a possible set of defining relations of $\Gamma(X_n)$ in $\SL$ is $\{ (T^{-1} \, R)^2 = R^n, R^{2n} = I, R^n T = T R^n \}$.\\
 
 The method of Schreier (see e.g. \cite{Joh97} Chapter 2) shows, that 
 $$SchreierGen = \{ R^n, T, R \, T \, ({R}^{n-1})^{-1}, \dots, {R}^{n-1} \, T \, {R}^{-1} \}$$
 is a free basis of an index $n$ subgroup $U$ of $F$ with Schreier transversal 
 $$\mathcal{T} = \{ I, R, \dots, {R}^{n-1}\} \,.$$
 
 Let $p: F \to \Gamma(X_n)$ denote the projection from $F$ onto $\Gamma(X_n)$. Our first claim is that $p(U) = G$.\\ 
 
 To simplify notation we skip the index $n$ in the following computations, i.e.\ $p(R) = R$ and $p(T) = T$.
 The inclusion $G \subset p(U)$ is immediate: The first two generators of $G$ belong to $p(SchreierGen)$, the other generators are contained in $\langle p(SchreierGen) \rangle$ since $R^j T^2 R^{-j} = R^j T (R^{n-j})^{-1} \cdot R^{n-j} T R^{-j}$.
 The inclusion $p(U) \subset G$ needs some more arithmetic.
 The relation $R^n =-I = (T^{-1} \, R)^2$ implies $T R^{-1} T = -R$ and $T^{-1} = -R^{-1} T R^{-1}$.
 It follows, that $R^j T (R^{n-j})^{-1} \in G$ for $1 \leq j \leq \frac{n-1}{2}$:
  $$ \begin{array}{cl}
      &(R^n)^j \cdot R^j T^2 R^{-j} \cdot R^{j-1} T^2 R^{-(j-1)} \cdot \dots \cdot R^2 T^2 R^{-2} \cdot R T^2 R^{-1} \cdot T \cdot (R^n)^{-1} \\
      =& \underbrace{(R^n)^j}_{(-I)^j} \cdot R^j T \underbrace{T R^{-1} T}_{-R} \underbrace{T R^{-1} \dots}_{-R} \dots \underbrace{\dots R^{-1} T}_{-R} \underbrace{T R^{-1} T}_{-R} \underbrace{T R^{-1} \cdot T}_{-R} \cdot R^{-n} \\
      =& (-I)^j (-I)^j R^j T R^{j-n} = R^j T (R^{n-j})^{-1}
     \end{array}$$
  If $\frac{n-1}{2} < j < n-1$ then $0 < n-1-j < \frac{n-1}{2}$, so that $R^{n-1-j} T (R^{n-(n-1-j)})^{-1} = R^{n-1-j} T R^{-j-1} \in G$. Because of   $$
     R^n \cdot (R^{n-1-j} T R^{-j-1})^{-1} 
     = -I \cdot R^{j+1} \underbrace{T^{-1}}_{-R^{-1} T R^{-1}} R^{-n+j+1}
     = R^j T R^{-n+j}
 $$
  this implies $R^j T (R^{n-j})^{-1} \in G$ for $\frac{n+1}{2} \leq j \leq n-2$.
  The equation $ R^n \cdot R^n \cdot T^{-1} = -I \cdot R^n \cdot (-R^{-1} T R^{-1}) = R^{n-1} T R ^{-1}$ completes the proof of the claim.\\
  
  It remains to show, that the kernel of $p$ is contained in $U$. Then the index $[\Gamma(X_n):G] = [F:U] = n$ with $\{ I, R, \dots, {R}^{n-1}\}$ a system of coset re\-pre\-sen\-ta\-tives. The kernel of the map $p$ is the normal closure of the set $Rel = \{ R^{n} (T^{-1} \, R)^{-2} , R^{2n}, R^n T R^{-n} T^{-1} \}$,
  so that we need to show $U w r = U w$ $\forall~w~\in~F$, $r \in Rel$. Because $\{ I, R, \dots, {R}^{n-1}\}$ is a system of coset representatives for $U$, this is equivalent to $R^j r R^{-j} \in U$ $\forall~r \in Rel, \, j \in \{1, \dots, n-1\}$.
  $$ \begin{array}{rcl}
      R^j \cdot R^n (T^{-1} \, R)^{-2} \cdot R^{-j} &=& R^{n + j - 1} T R^{-1} T R^{-j}\\
      &=& R^n \cdot R^{j-1} T (R^{n-(j-1)})^{-1} \cdot R^{n-j} T R^{-j} \in U \\ \noalign{\medskip}
      R^j \cdot R^{2n} \cdot R^{-j} &=& R^{2n} = (R^n)^2 \in U \\ \noalign{\medskip}
      R^j \cdot R^n T R^{-n} T^{-1} \cdot R^{-j} &=& R^n \cdot R^j T (R^{n-j})^{-1} R^{-n} R^{n-j} T^{-1} R^{-j} \\
      &=& R^n \cdot R^j T (R^{n-j})^{-1} \cdot R^{-n} \cdot (R^j T R^{-n+j})^{-1} \in U
     \end{array}
  $$ 
\end{proof}

The surface $X_n$ is primitive and has genus greater than one, so
$\Gamma(Y_{n,d})$ is a subgroup of $\Gamma(X_n)$ (see Lemma~\ref{Lemma:prim}). In combination with Lemma~\ref{Lemma:VeechGroupGen} the proof of Theorem 1 a) 
reduces to the proof of
 \begin{enumerate}[(i)]
 \item $\{ -I, T_n, R_n \, {T_n}^2 \, R_n^{-1}, \dots, {R_n}^{\frac{n-1}{2}} \ {T_n}^2 \, {R_n}^{-\frac{n-1}{2}} \} \subset \Gamma_n = \Gamma(Y_{n,d})$ and
  \item $I, R_n, \dots, {R_n}^{n-1}$ lie in different cosets of $\Gamma_n \bs \Gamma(X_n)$.
 \end{enumerate}

The rest of the proof of Theorem 1 a)
is divided into four steps. 
The first step is to prove, that the parabolic matrices $R_n \, {T_n}^2 \, {R_n}^{-1}, \dots, {R_n}^{\frac{n-1}{2}} \ {T_n}^2 \, {R_n}^{-\frac{n-1}{2}}$ with shearing factor $2 \lambda_n$ are contained in $\Gamma_n$. 
Afterwards we show, that $I$, $R_n, \dots, {R_n}^{n-1}$ lie in different cosets of $\Gamma_n \bs \Gamma(X_n)$.
The last two steps will show that $-I$ and $T_n$ belong to $\Gamma_n$.

\subsubsection*{Step 1: \boldmath{${R_n}^l \, {T_n}^2 \, {R_n}^{-l} \in \Gamma_n$} for \boldmath{$1 \leq l \leq \frac{n-1}{2} = k_1$}}
The affine map on $\R^2$, $x \mapsto M_l \cdot x$ with
$$M_l = {R_n}^l \, {T_n}^2 \, {R_n}^{-l} 
= \begin{pmatrix}
1 - 2 \lambda_n \cos(\frac{l \pi}{n}) \sin(\frac{l \pi}{n}) & 2 \lambda_n \cos^2(\frac{l \pi}{n}) \\ \noalign{\medskip}
-2 \lambda_n \sin^2(\frac{l \pi}{n}) & 1 + 2 \lambda_n \cos(\frac{l \pi}{n}) \sin (\frac{l \pi}{n})
\end{pmatrix}
$$
is the shearing in direction $v_l = {R_n}^l \cdot \begin{pmatrix} 1 \\ 0 \end{pmatrix} = \begin{pmatrix} \cos(l \pi /n) \\ \sin(l \pi /n) \end{pmatrix}$ with factor $2 \lambda_n$.
To find an affine map on $Y_{n,d}$ with derivative $M_l$, we investigate the cylinder decomposition of $Y_{n,d}$ in direction $v_l$. \\

Because of the rotational symmetry in $X_n$ the cylinder decomposition of $X_n$ in direction $v_l$ has the same properties as the horizontal cylinder decomposition (see page \pageref{fig:CylinderSizes}).
The vector $v_l$ lies in the same direction as the edge $x_{l/2}$ if $l$ is even and in the same direction as the edge $x_{k_1 + (l+1)/2}$, if $l$ is odd. The cylinders in the cylinder decomposition in direction $x_i$ contain gluings along the following edges (where all the indices are understood to be modulo $n$):
\begin{itemize}
 \item cylinder $1$: $x_{i-1}$ and $x_{i+1}$
 \item cylinder $2$: $x_{i-2}$ and $x_{i+2}$\\
      $\vdots$
 \item cylinder $\frac{n-1}{2}$: $x_{i-\frac{n-1}{2}}$ and $x_{i+\frac{n-1}{2}}$
\end{itemize}

Since the monodromy map for $Y_{n,d}$ sends $x_i$ to the identity in $S_d$ iff $i \notin\{k_1, k_2\}$, the $d$ copies of $X_n$ are connected only by the edges $x_{k_1}$ and $x_{k_2}$. Copies of cylinders in $X_n$ that contain neither $x_{k_1}$ nor $x_{k_2}$ are glued to themselves in $Y_{n,d}$. They keep their inverse modulus $\lambda_n$.
Cylinders that contain either $x_{k_1}$ or $x_{k_2}$ are glued according to $m_{n,d} (x_{k_1})$ or to $m_{n,d}(x_{k_2})$ respectively. For every cycle of length $a$ in $m_{n,d} (x_{k_1})$ (or in $m_{n,d}(x_{k_2})$), $a$ copies of the cylinder containing $x_{k_1}$ (or $x_{k_2}$) are glued to form one cylinder in $Y_{n,d}$ (see Figure \ref{fig:CylinderDiagonal}). Since both $m_{n,d} (x_{k_1})$ and $m_{n,d}(x_{k_2})$ have only cycles of length $2$, the inverse modulus of theses cylinders is $2\lambda_n$.
Non of the vectors $v_l$ is horizontal ($l \neq 0$), so that no cylinder decomposition in direction $v_l$ has a cylinder containing both edges $x_{k_1}$ and $x_{k_2}$.

\begin{figure}[htb]
 \centering
  \includegraphics[width=0.5\textwidth]{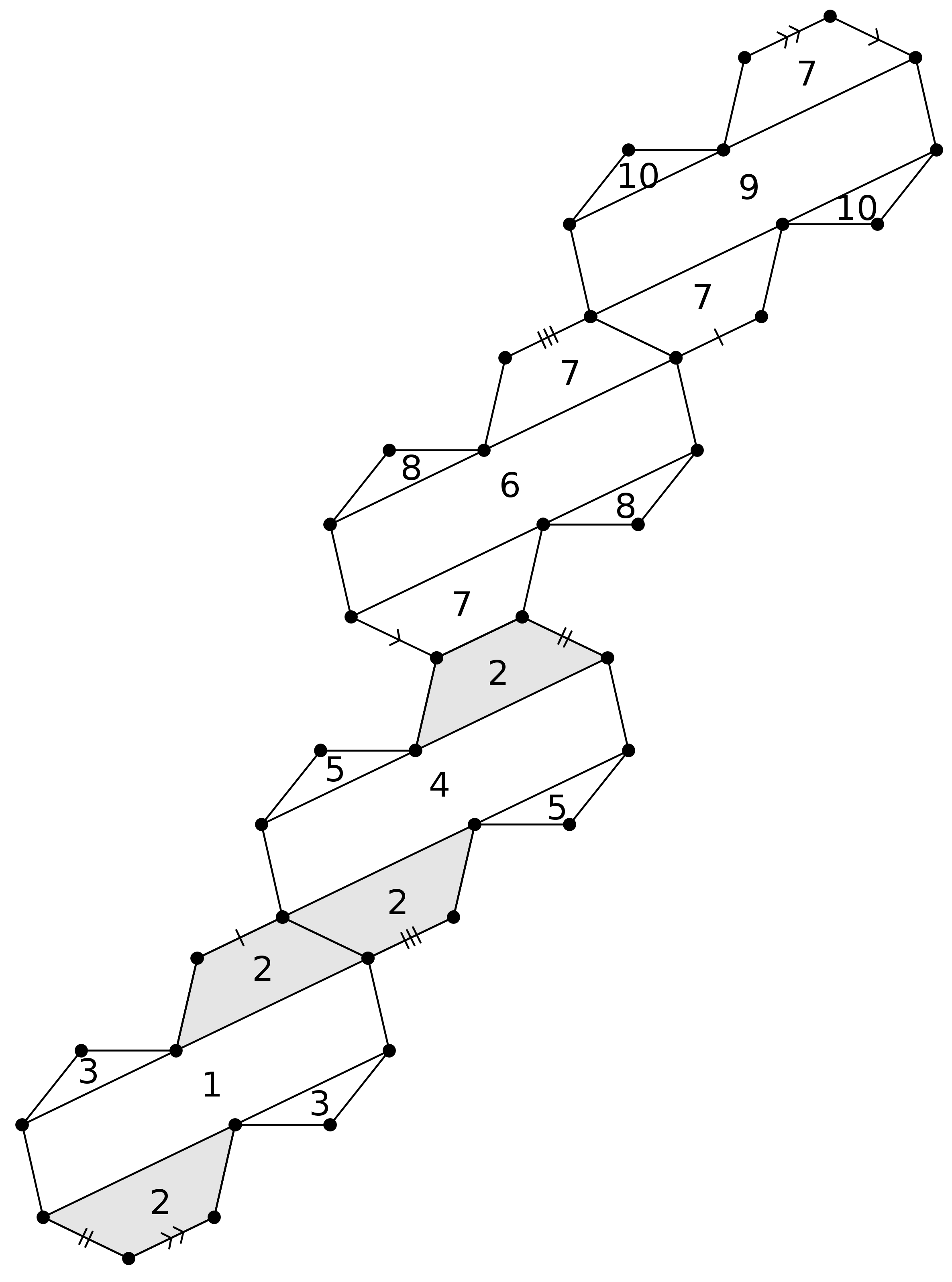}
 \caption{Diagonal cylinder decomposition}
 \label{fig:CylinderDiagonal}
\end{figure}

We may conclude, that $M_l$ is the derivative of the affine map $\phi_l$ on $Y_{n,d}$, fixing each saddle connection in direction $v_l$ pointwise and twisting the cylinders in the cylinder decomposition in direction $v_l$ once or twice (again see e.g. \cite{Vor96} Chapter 3.2 for a more detailed explanation of such an affine map).
It follows that ${R_n}^l \, {T_n}^2 \, {R_n}^{-l} \in \Gamma_n$ for $1 \leq l \leq k_1$.

\subsubsection*{Step 2: \boldmath{$I, R_n, \dots, {R_n}^{n-1}$} lie in different cosets of \boldmath{$\Gamma_n \bs \Gamma(X_n)$}}

It is sufficient to show, that ${R_n}^l \notin \Gamma_n \; \forall \, l \in \{1, \dots, n-1\}$, since $$\Gamma(X_n)\cdot {R_n}^i = \Gamma(X_n)\cdot{R_n}^j \Leftrightarrow {R_n}^{i-j} \in \Gamma(X_n)\,.$$
In this part of the proof we use the horizontal cylinder decomposition of $X_n$. In this decomposition, cylinder $k_1 = \frac{n-1}{2}$ contains the gluings along the edges $x_{k_1}$ and $x_{k_2}$. More precisely, a closed horizontal path in the cylinder $k_1$ describes an element in $\langle x_{k_1} {x_{k_2}}^{-1} \rangle \subset \pi_1(X_n \setminus \Sigma)$.
As above we conclude that for every cycle of length $a$ in $m_{n,d}(x_{k_1} {x_{k_2}}^{-1})$ $a$ copies of the cylinder $k_1$ are glued in $Y_{n,d}$ to form one cylinder.\\

If $d$ is even, then:
$$\begin{array}{l}
   m_{n,d}(x_{k_1} {x_{k_2}}^{-1}) = m_{n,d}({x_{k_2}}^{-1}) \circ m_{n,d}(x_{k_1}) \\
   = (\; (1 \; 2) \, (3 \; 4) \, \cdots \, (d-3 \;\; d-2)\,(d-1 \;\; 0) \;)^{-1} \circ (0 \; 1) \, (2 \; 3) \, \cdots \, (d-2 \;\; d-1)\\
   = (1 \; 2) \, (3 \; 4) \, \cdots \, (d-3 \;\; d-2) \, (d-1 \;\; 0) \circ (0 \; 1) \, (2 \; 3) \, \cdots \, (d-2 \;\; d-1)\\
   = (0 \;\; 2 \;\; 4 \dots d-2)\,(1\;\;d-1\;\;d-3 \dots 3)
  \end{array}$$
If $d$ is odd, it follows:
$$\begin{array}{l}
   m_{n,d}(x_{k_1} {x_{k_2}}^{-1}) = m_{n,d}({x_{k_2}}^{-1}) \circ m_{n,d}(x_{k_1}) \\
   = (1 \; 2) \, (3 \; 4) \, \cdots \, (d-2 \;\; d-1) \circ (0 \; 1) \, (2 \; 3) \, \cdots \, (d-3 \;\; d-2)\\
   = (0 \;\; 2 \;\; 4 \dots d-3 \;\; d-1\;\;d-2\;\;d-4 \dots 3\;\;1)
  \end{array}$$

Therefore the inverse modulus of the horizontal cylinders of $Y_{n,d}$, which map onto the cylinder $k_1$ of $X_n$, is $\frac{d}{2} \lambda_n$ if $d$ is even and $d \lambda_n$ if $d$ is odd. All the other horizontal cylinders in $Y_{n,d}$ have inverse modulus $\lambda_n$.

The conclusions made in step 1 about the cylinders in the direction $v_l$ do not rely on $l$ to be smaller or equal to $\frac{n-1}{2}$. They also apply for $1 \leq l \leq n-1$. We conclude that all the cylinders in direction $v_l$ have inverse modulus $\lambda_n$ or $2 \lambda_n$. 
Now suppose $\phi_{R^l}$ to be an affine map on $Y_{n,d}$ with derivative ${R_n}^l$. Then $\phi_{R^l}$ would send the horizontal cylinders onto the cylinders in direction $v_l$. As a rotation is length preserving, the moduli of the cylinders in horizontal direction would have to match the moduli of the ones in direction $v_l$.
It is an immediate consequence for $d \notin \{2,4\}$, that ${R_n}^l$ is not contained in $\Gamma_n$. Simply, because there is no cylinder in direction $v_l$ with inverse modulus $\frac{d}{2} \lambda_n$ or $d \lambda_n$ respectively.\\

For $d=2$, all the cylinders in horizontal direction have inverse modulus $\lambda_n$. In all other directions $v_l$, there is no cylinder containing both edges $x_{k_1}$ and $x_{k_2}$. On the other hand, only one of the edges $x_{k_1}$ or $x_{k_2}$ can lie in direction $v_l$, so that there is at least one cylinder containing one of $x_{k_1}$ or $x_{k_2}$. This cylinder has only one preimage cylinder $\tilde{c}$ in $Y_{n,d}$, so that the cylinder $\tilde{c}$ has inverse modulus $2 \lambda_n$, which leads to a contradiction.\\

If $d=4$, then there are $2$ cylinders with inverse modulus $2 \lambda_n$ in the horizontal decomposition of $Y_{n,d}$. For all directions $v_l$, where $x_{k_1}$ and $x_{k_2}$ do not lie in direction $v_l$, $Y_{n,d}$ has 4 cylinders with inverse modulus $2 \lambda_n$, $2$ from $x_{k_1}$ and $2$ from $x_{k_2}$. 
In the remaining cases, were $d=4$ and $v_l$ is parallel to $x_{k_1}$ or $x_{k_2}$, the moduli of the cylinders in the decomposition in direction $v_l$ match the ones of the horizontal decomposition. Here we have to examine the heights of the cylinders with inverse modulus $2 \lambda_n$. 
The cylinder $k_1 = \frac{n-1}{2}$ and of course also its preimages, have the height $2 \sin(\frac{\pi}{n}) \sin(\frac{\pi}{n})$ (see equation~\ref{eq:height} on page~\pageref{eq:height}). The cylinder of $X_n$ in direction $v_l$ which is parallel to $x_{k_1}$ or $x_{k_2}$ and contains the gluing along $x_{k_2}$ or $x_{k_1}$ respectively, is the cylinder $1$ (i.e.\ the image of cylinder $1$ in the horizontal decomposition under ${R_n}^l$). It has the height $2 \sin((n-2)\frac{\pi}{n}) \sin(\frac{\pi}{n})$. Using the consequence of the addition and subtraction theorems for $\sin$ and $\cos$ 
$$\sin x - \sin y = 2 \cos (\frac{x+y}{2}) \sin(\frac{x-y}{2})\;,$$
the difference of the heights of the cylinders with inverse modulus $2 \lambda_n$ in $Y_{n,d}$ can be simplified as follows:
$$\begin{array}{l}
    2 \sin(\frac{\pi}{n}) \sin(\frac{\pi}{n}) - 2 \sin(\frac{\pi}{n}) \sin((n-2 ) \frac{\pi}{n}) \\
   = 2 \sin(\frac{\pi}{n}) \cos(\frac{n-1}{2} \cdot \frac{\pi}{n}) \sin(\frac{-n+3}{2} \cdot \frac{\pi}{n}) \quad \neq 0
  \end{array}
$$
Since a rotation can not map cylinders with different heights onto each other, this case also leads to a contradiction.

\subsubsection*{Step 3: \boldmath{$T_n \in \Gamma_n$}}

To show, that $T_n$ is contained in the Veech group of $X_n$, we use again the horizontal cylinder decomposition of $X_n$. In the following we construct an affine map $\phi_T$ on $Y_{n,d}$ with derivative $T_n$.\\

As shown above, for all but the outermost cylinder of $X_n$ the preimage on $Y_{n,d}$ decomposes into $d$ cylinders with inverse modulus $\lambda_n$, respectively. 
An affine map $\phi_T$ with derivative $T_n$ on $Y_{n,d}$ twists these cylinders once and possibly permutes them, with say $\sigma_T$. Of course, if an inner cylinder $j$ in copy $i$ is mapped onto cylinder $j$ in copy $\sigma_T(i)$, then all the other inner cylinders in copy $i$ are mapped onto their correspondent in copy $\sigma_T(i)$.
The preimages of cylinder $\frac{n-1}{2}$ in $X_n$ (now called $c$) form $1$ or $2$ cylinders in $Y_{n,d}$, depending on $d$ being odd or even. These cylinders are sheared by the factor $\lambda_n$.

Recall, that $Y_{n,d}$ is a degree $d$ translation covering of $X_n$, in particular $Y_{n,d}$ contains $d$ copies of every point in $X_n$, except for the vertices which may be ramification points. The preimage of $c$ in one copy of $X_n$ is not connected, so from now on we use a slightly different labelling. To get the new labelling we cut the lower part of $c$ in each copy along the horizontal saddle connection and glue it along the edge $x_{k_1}$, i.e.\ the lower part of $c$ in copy $i$ now ``belongs'' to the copy $m_{n,d}({x_{k_1}}^{-1})(i)$ (see Figure \ref{fig:Label_C}).

\begin{figure}[htb]
 \centering
  \includegraphics[width=0.8\textwidth]{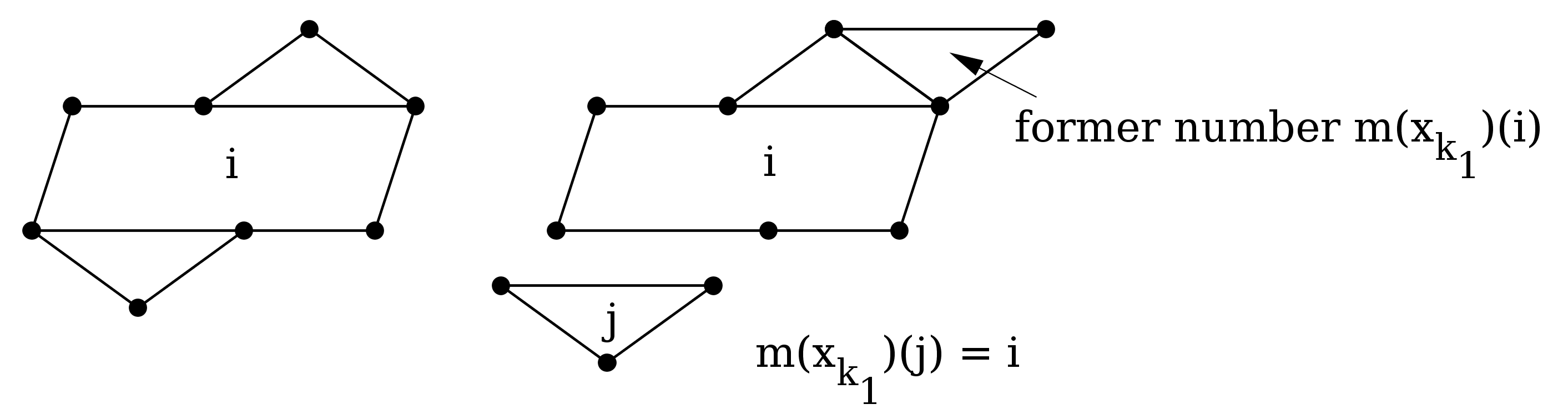}
 \caption{New labelling of the outermost cylinders}
 \label{fig:Label_C}
\end{figure}

Every new copy $i$ of $c$ has as lower neighbour copy $i$ of $X_n \setminus c$ and as upper neighbour copy $m_{n,d}(x_{k_1})(i)$ of $X_n \setminus c$ (see Figure \ref{fig:Phi_T}).

\begin{figure}[htb]
 \centering
  \includegraphics[width=0.7\textwidth]{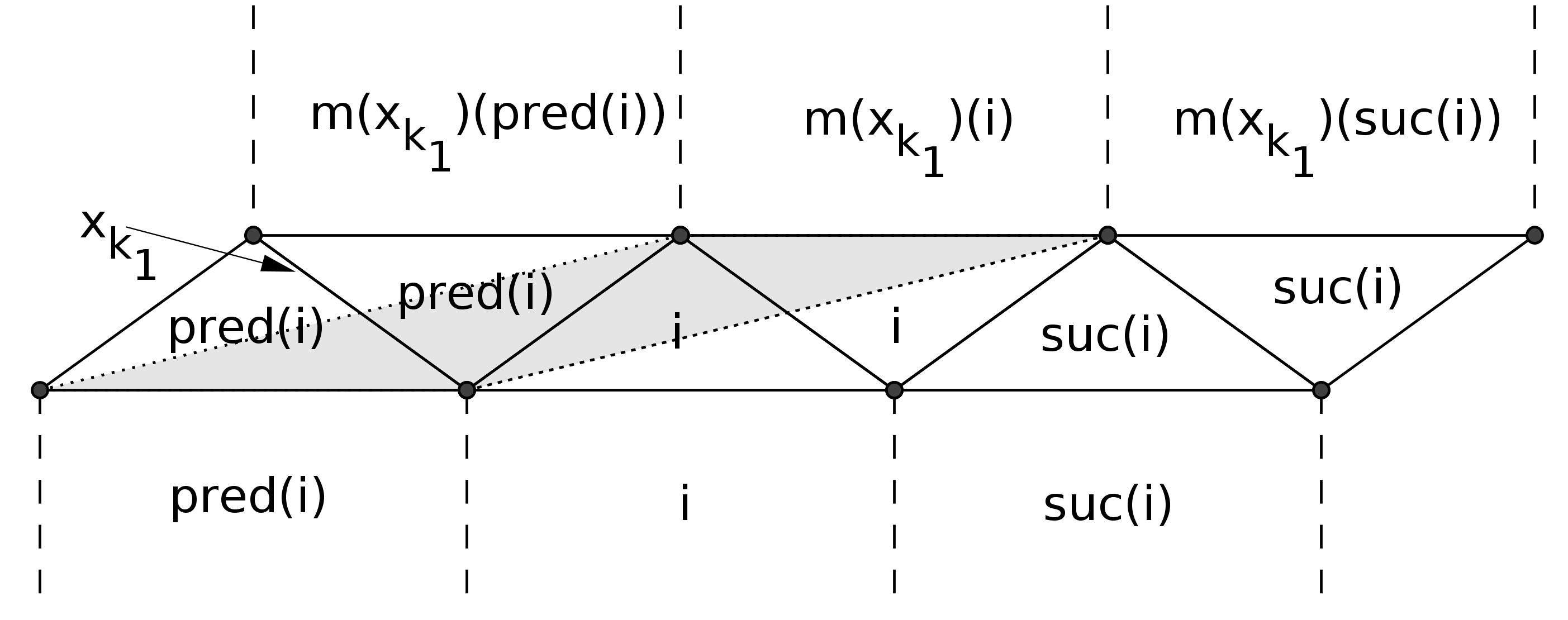}
 \caption{Schematic drawing of the outermost cylinder}
 \label{fig:Phi_T}
\end{figure}

The preimage cylinder(s) of $c$ define a natural order of the $d$ copies:
$$\suc(i) = m_{n,d}(x_{k_1} {x_{k_2}}^{-1})(i) \qquad \pred(i) = m_{n,d}(x_{k_2} {x_{k_1}}^{-1})(i)$$

All we need to prove is the existence of a permutation $\sigma_T$ of the copies of $X_{n,d}$ that fulfils the following two conditions:
\begin{enumerate}
 \item The lower neighbours retain their order, i.e.\ the image of the successor of $i$ equals the successor of the image of $i$:
	$$\sigma_T(\suc(i)) = \suc(\sigma_T(i))$$
 \item The upper neighbours are shifted by one, i.e.\ the upper neighbour of the image of $i$ is the image of the upper neighbour of the predecessor of $i$:
	$$m_{n,d}(x_{k_1}) (\sigma_T(i)) = \sigma_T(m_{n,d}(x_{k_1})(\pred(i)))$$
	Using the definition of $\pred$ and $m_{n,d}(a \cdot b) = m_{n,d}(b) \circ m_{n,d}(a)$, the equation can be written as 
	$$m_{n,d}(x_{k_1}) (\sigma_T(i)) = \sigma_T(m_{n,d}(x_{k_2})(i)) \,. $$
\end{enumerate}
Such a permutation induces an affine map $\phi_T$ as follows:\\
The inner cylinders of copy $i$ are mapped to the inner cylinders of copy $\sigma_T(i)$ and further twisted once. The preimages of $c$ are sheared by the factor $\lambda_n$. The lower bounding saddle connection of copy $i$ of $c$ is mapped onto the corresponding saddle connection in copy $\sigma_T(i)$, the upper saddle connection to its correspondent in copy $\sigma_T(\suc(i))$ (in the new labelling). The two conditions above ensure, that $\phi_T$ respects the gluings.

\begin{claim}
 For even $d$, $\sigma_T = (1 \;\; 3\;\;5 \dots d-3 \;\;d-1)$ satisfies condition 1 and 2.
\end{claim}
 \begin{proof}
 $$ \begin{array}{rcl}
   \sigma_T \circ \suc &=& (1 \;\; 3\;\;5 \dots d-3 \;\;d-1) \circ (0 \;\; 2 \;\; 4 \dots d-2)\,(1\;\;d-1\;\;d-3 \dots 3)\\
   &=& (0 \;\; 2 \;\; 4 \dots d-2)\\
   &=& (0 \;\; 2 \;\; 4 \dots d-2)\,(1\;\;d-1\;\;d-3 \dots 3) \circ (1 \;\; 3\;\;5 \dots d-3 \;\;d-1) \\
   &=& \suc \circ \sigma_T    
   \end{array}$$
  $$ \begin{array}{rcl}
     m_{n,d}(x_{k_1}) \circ \sigma_T 
     &=& (0 \; 1) \, (2 \; 3) \, \cdots \, (d-2 \;\; d-1) \circ (1 \;\; 3\;\;5 \dots d-3 \;\;d-1)\\
     &=& (0\;\;1\;\;2\;\;3\dots d-2\;\;d-1)\\
     &=& (1 \;\; 3 \dots d-3 \;\;d-1) \circ (0 \;\; d-1) \, (1 \; 2) \, \cdots \, (d-3 \;\; d-2)\\
     &=& \sigma_T \circ m_{n,d}(x_{k_2})
   \end{array}$$
 \end{proof}

\begin{claim}
 For odd $d$, $\sigma_T = m((x_{k_1} x_{k_2})^\frac{d-1}{2})$ satisfies condition 1 and 2.
\end{claim}
 \begin{proof}
 Recall, that $m_{n,d}(x_{k_1}) = \sigma_{d,1}$, $m_{n,d}(x_{k_2})= \sigma_{d,2}$ and that $\sigma_{d,1}$ as well as $\sigma_{d,2}$ is of order $2$.
 Since the permutation $m_{n,d}(x_{k_1} x_{k_2}) = \sigma_{d,2} \circ \sigma_{d,1} = (0 \;\; 2 \;\; 4 \dots d-3 \;\; d-1\;\;d-2\;\;d-4 \dots 3\;\;1)$ is of order $d$, it follows that $(\sigma_{d,2} \circ \sigma_{d,1})^d = \id$ and the second condition $$\sigma_{d,1} \circ (\sigma_{d,2} \circ \sigma_{d,1})^\frac{d-1}{2} = (\sigma_{d,2} \circ \sigma_{d,1})^d \circ \sigma_{d,1} \circ (\sigma_{d,2} \circ \sigma_{d,1})^\frac{d-1}{2} = (\sigma_{d,2} \circ \sigma_{d,1})^\frac{d-1}{2} \circ \sigma_{d,2}$$ holds.
 The first condition is immediate:
 $$\sigma_T \circ \suc = (\sigma_{d,2} \sigma_{d,1})^\frac{d-1}{2} \circ {\sigma_{d,2}}^{-1} \sigma_{d,1} = (\sigma_{d,2} \sigma_{d,1})^\frac{d+1}{2} = \suc \circ \sigma_T$$
 \end{proof}
The two claims finish the proof of the fact that $T_n \in \Gamma_n$.
 
\subsubsection*{Step 4: \boldmath{$-I \in \Gamma_n$}}
We need to construct an affine map $\phi_{-I}$ that is locally a rotation by $\pi$. Let $\phi_{-I}$ denote the map that rotates every copy of $X_n$ around the centre of its edge $x_{n-1}$ by $\pi$. Clearly the map sends every copy of $X_n$ to itself. The preimage of an edge $x_i$ is send to $x_i'$. It remains to show, that this is consistent with the gluing of the copies according to the monodromy map. The monodromy $m_{n,d}(x_i)$ for $i \notin \{k_1,k_2\}$ is trivial, so there is nothing to prove. The permutations $m_{n,d}(x_{k_1})$ and $m_{n,d}(x_{k_2})$ are involutions, so that the neighbour copy at $x_i$ equals the neighbour copy at $x_i'$ for all $i \in \{0, \dots n-2\}$. It follows, that $-I \in \Gamma_n$.\\

Step 1 to 4 together with Lemma \ref{Lemma:VeechGroupGen} and Lemma \ref{Lemma:prim} prove Theorem 1 a).\\

Knowing the Veech group of the translation surfaces $Y_{n,d}$ for all odd $n \geq 5$ and all $d \geq 2$, we can compute the basic properties of the Teichm\"uller curves belonging to the surfaces.

\begin{corollary}
\label{cor:Curve}
 $\HH / \Gamma_n$ has genus $0$ and $\frac{n+1}{2}$ cusps.
\end{corollary}
\begin{proof}
 A fundamental region $F$ of $\Gamma(X_n)$ is shown in Figure~\ref{fig:FundX5}. 
 Since $\Gamma(X_n) = \Gamma_n \, I \cup \Gamma_n \, R_n \cup \dots \cup \Gamma_n \, {R_n}^{n-1}$, the fundamental region of $\Gamma_n$ is $G = I(F) \cup R_n(F) \cup \dots \cup {R_n}^{n-1}(F)$ (see e.g. \cite{Kat92} Theorem 3.1.2). 
Figure~\ref{fig:FundamentalGamma} illustrates the fundamental region $G$ of $\Gamma_5$ in a schematic picture.
\begin{figure}[htbp]
\centering
\subfigure[$\Gamma(X_n) = \langle T_n, R_n \rangle$]{
\label{fig:FundX5}
\centering
\includegraphics[width=0.35\textwidth]{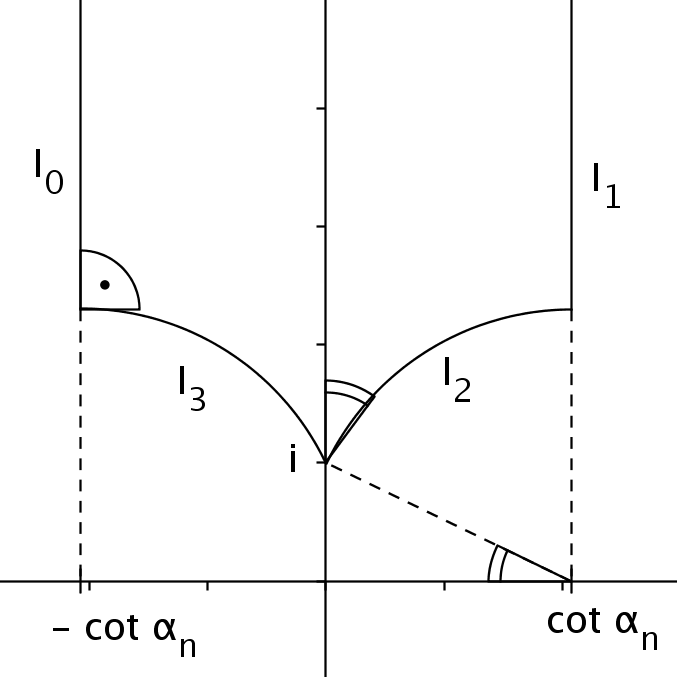}
}
\subfigure[$\Gamma_5$]{
 \label{fig:FundamentalGamma}
\centering
 \includegraphics[width=0.35\textwidth]{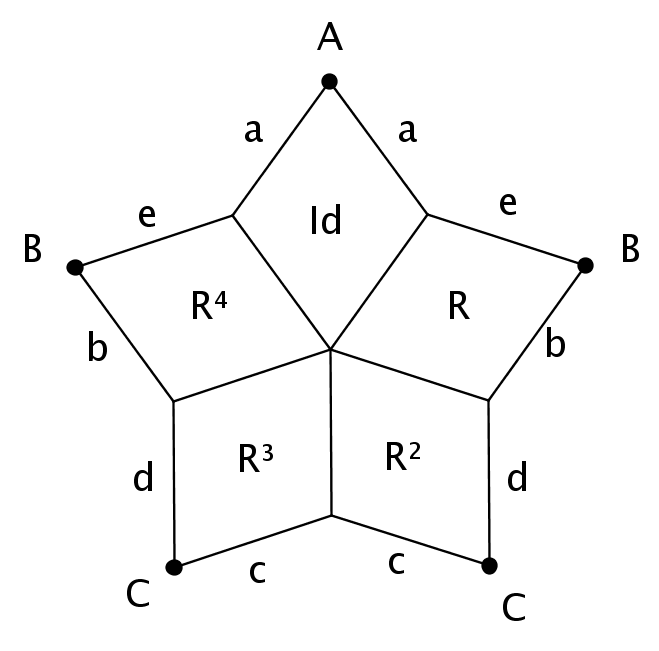}
}
 \caption{Fundamental regions}
 \label{fig:FundamentalbereichX5}
\end{figure}
The elements $T_5$, $R_5 T_5 {R_5}^{-4}$, ${R_5}^2 T_5 {R_5}^{-3}$, ${R_5}^3 T_5 {R_5}^{-2}$ and ${R_5}^4 T_5 {R_5}^{-1}$ identify the edges with labels $a$, $b$, $c$, $d$ and $e$. The cusps in $\R \cup \infty$, i.e.\ the non equivalent vertices of the fundamental region of $\Gamma_5$, are labelled by the letters $A$ to $C$, where $A = \{\infty\}$, $B = \{\cot \alpha_5, -\cot \alpha_5\}$ and $C = \{ \cot 2 \alpha_5, - \cot 2 \alpha_5 \}$, with $\alpha_n = \frac{\pi}{n}$. 
At the three cusps in $\HH / \Gamma_5$, one, two and two copies of the fundamental region of $\Gamma(X_5)$ meet, respectively. We say the cusps have \textit{relative width} $1$, $2$ and $2$. In addition it can be seen, that the Euler characteristic $\chi(\HH / \Gamma_5) = (3 + 4) - (2 \cdot 5) + 5 = 2$ and thus the genus of $\HH / \Gamma_5$ is $0$.\\

In general, the fundamental region $G$ of $\Gamma_n$ consists of $n$ quadrilaterals with one cusp $A_0 = \{ \infty\}$ of relative width $1$ and $\frac{n-1}{2}$ cusps $A_i = \{\cot i \alpha_n, -\cot i \alpha_n\}$ for $i = 1, \dots, \frac{n-1}{2}$ of relative width $2$. In addition, $\HH / \Gamma_n$ has a vertex at $i$ and $\frac{n+1}{2}$ non equivalent vertices at 
 $\cot \alpha_n + \frac{i}{\sin \alpha_n}$, $R (\cot \alpha_n + \frac{i}{\sin \alpha_n})$, $\dots$, $R^\frac{n-1}{2} (\cot \alpha_n + \frac{i}{\sin \alpha_n})$.
It follows, that $\chi(\HH / \Gamma_n) = (\frac{n+1}{2} + 1 + \frac{n+1}{2}) - (2 \cdot n) + n = 2$ which implies $g(\HH / \Gamma_n) = 0$.
\end{proof}

\subsection{The Veech group of the even series.}
In this section, we explicitly determine the Veech group for $Y_{n,d}$ in the case that $n$ is even. Just like in Section \ref{sec:VeechUeberlagerungOdd} the Veech group depends only on $n$ and is generated by $-I$ and parabolic matrices. Hence we obtain Theorem~1~b) 
as analogue of Theorem~1~a). 
\begin{ThmEven}
For even $n\geq 8$, the Veech group of $Y_{n,d}$ is
 $$\begin{array}{rl}
    \Gamma_n :=& \Gamma(Y_{n,d})\\
    = &\langle -I, T_n, {R_n}^2 \, {T_n}^2 \, {R_n}^{-2}, \dots, {R_n}^{n-2} \, {T_n}^2 \, {R_n}^{-(n-2)}, \, ({T_n}^{-1}{R_n}^2)^2, \\ \noalign{\smallskip}
     & {R_n}^2 \,({T_n}^{-1}{R_n}^2)^2 \, {R_n}^{-2}, \dots, {R_n}^{n-2} \, ({T_n}^{-1}{R_n}^2)^2 \, {R_n}^{-(n-2)} \rangle\;.
   \end{array} $$
The matrices $R_n$ and $T_n$ are defined as in Chapter \ref{sec:nGon} 
as
$$R_n = \begin{pmatrix} \cos{\frac{\pi}{n}} & -\sin{\frac{\pi}{n}} \\ \sin{\frac{\pi}{n}} & \cos{\frac{\pi}{n}} \end{pmatrix} \; \textrm{and } \; T_n = \begin{pmatrix} 1 & \lambda_n \\ 0 & 1 \end{pmatrix} \; \textrm{where } \lambda_n= 2 \cot{\frac{\pi}{n}}$$
and $I \in \GL$ is the identity matrix.
\end{ThmEven}
The proof is very similar to the proof of Theorem 1 a). Hence we will state the outline of the proof while mainly referring to the corresponding proofs in Section \ref{sec:VeechUeberlagerungOdd}.
\begin{proof}[Concept of proof:]
 Using similar arguments as in Lemma \ref{Lemma:VeechGroupGen} one obtains that $\Gamma_n$ is an index $\frac{n}{2}$ subgroup of $\Gamma(X_n) = \langle {R_n}^2, T_n\rangle$ with coset representatives $I, {R_n}^2, {R_n}^4, \dots, {R_n}^{n-2}$.
 We now have to repeat the steps 1 to 4. Step 3 and 4 work the same way as in Section \ref{sec:VeechUeberlagerungOdd}. Thus it is sufficient to carry out step 1 and step 2.

\paragraph{Step 1:}
We have to find affine maps with parabolic derivatives ${R_n}^{2l} \, {T_n}^2 \, {R_n}^{-2l}$ for $l \in \{1, \dots, \frac{n-2}{2}\}$ and ${R_n}^{2l} \, ({T_n}^{-1}{R_n}^2)^2 \, {R_n}^{-2l}$ for $l \in \{0, \dots, \frac{n-2}{2}\}$. 
Let $v_j = {R_n}^j \cdot \begin{pmatrix} 1 \\ 0 \end{pmatrix}$.
The matrix ${R_n}^{2l} \, {T_n}^2 \, {R_n}^{-2l}$ is the derivative of the shearing in direction $v_{2l}$ with factor $2 \lambda_n$.
A consequence of the relation $R_n {T_n}^{-1} R_n = -T_n$ is $({T_n}^{-1}{R_n}^2)^2 = {R_n}^{-1} {T_n}^2 R_n$, so
the matrix ${R_n}^{2l} \, ({T_n}^{-1}{R_n}^2)^2 \, {R_n}^{-2l}$ 
is the derivative of the shearing in direction $v_{2l-1}$ with factor $2 \lambda_n$. 
Hence we investigate the cylinder decompositions of $X_n$ and $Y_{n,d}$ in the directions $v_j$ for $j \in \{-1,1,2, \dots, n-2 \}$.
In direction $v_{2l-1}$, $X_n$ decomposes into $\frac{n}{4}$ cylinders if $n \equiv 0 \mmod 4$ and into $\frac{n+2}{4}$ cylinders, if $n \equiv 2 \mmod 4$. A short calculation shows that the innermost cylinder has inverse modulus $\frac{1}{2} \lambda_n$ and that all other cylinders have inverse modulus $\lambda_n$. 
Recall that the cylinders in the cylinder decomposition in direction $v_{2l}$ have inverse modulus $\lambda_n$ (see equation~\ref{eq:widthAndHeightEven} on page~\pageref{eq:widthAndHeightEven}).
Only the vertical decomposition ($j = \frac{n}{2}$) contains a cylinder containing both $x_{k_1}$ and $x_{k_2}$ ($j \neq 0$ excludes the horizontal direction).
We conclude that the arguments in the odd case also hold in the even case for the parabolic matrices 
${R_n}^{2l} \, {T_n}^2 \, {R_n}^{-2l}$ where $l \in \{1, \dots, \frac{n-2}{2}\}$, $l \neq \frac{n}{4}$ and for ${R_n}^{2l} \,({T_n}^{-1}{R_n}^2)^2 \, {R_n}^{-2l}$ where $l \in \{0, \dots, \frac{n-2}{2}\}$, $l \neq \frac{n+2}{4}$.

For $n \equiv 0 \mmod 4$ and $l = \frac{n}{4}$, the innermost cylinder in the decomposition of $X_n$ in direction $v_{2l}$ contains $x_{k_1}$ and $x_{k_2}$. This cylinder has $1$ or $2$ preimage cylinders in $Y_{n,d}$ with inverse modulus $d \lambda_n$ or $\frac{d}{2} \lambda_n$, respectively (depending on $d$ being odd or even). If we rotate $Y_{n,d}$ and $X_n$ clockwise by 90 degrees and introduce a new labelling of the preimages of $X_n$ in $Y_{n,d}$ (see Figure \ref{fig:Label_X8New}), we can repeat the arguments of step 3 in Section \ref{sec:VeechUeberlagerungOdd} and find an affine map for $M = {R_n}^{\frac{n}{2}} \, T_n \, {R_n}^{-\frac{n}{2}}$.
\begin{figure}[htb]
 \centering
  \includegraphics{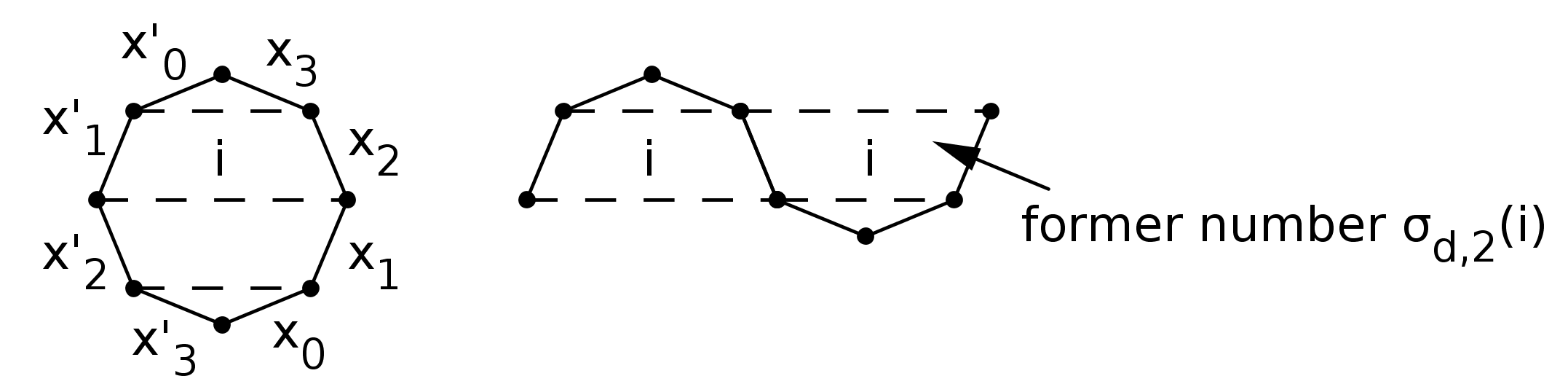}
 \caption{New labelling of the cylinders for $X_8$}
 \label{fig:Label_X8New}
\end{figure}
Then of course $M^2 = {R_n}^{\frac{n}{2}} \, {T_n}^2 \, {R_n}^{-\frac{n}{2}}$ lies in $\Gamma(Y_{n,d})$.
To find an affine map for $M$ we have to find a permutation $\sigma_M$ of the copies of $X_{n}$ in $Y_ {n,d}$ that fulfils the two conditions $\sigma_M(\suc(i)) = \suc(\sigma_M(i))$ and $\sigma_{d,2} (\sigma_M(i)) = \sigma_M(\sigma_{d,2}(\pred(i)))$ with $\suc(i)= \sigma_{d,1}(\sigma_{d,2}(i))$ and $\pred = \suc^{-1}$. The permutation $\sigma_T^{-1}$ with $\sigma_T$ as in step 3 in Section \ref{sec:VeechUeberlagerungOdd} fulfils these two conditions.

The same arguments prove that $M^2= {R_n}^{\frac{n+2}{2}} \, ({T_n}^{-1}{R_n}^2)^2 \, {R_n}^{-\frac{n+2}{2}}$ is contained in the Veech group of $Y_{n,d}$ if $n \equiv 2 \mmod 4$ ($l = \frac{n+2}{4}$).

\paragraph{Step 2:}
In the same way as in Section \ref{sec:VeechUeberlagerungOdd} we obtain that ${R_n}^{2l} \notin \Gamma(Y_{n,d})$ for $l \in \{1, \dots, \frac{n-2}{2}\}$, $l \neq \frac{n}{4}$ if $d \neq 4$ or if $d=4$ and $v_l = {R_n}^{2l} \cdot \begin{pmatrix} 1 \\ 0 \end{pmatrix}$ is not parallel to $x_{k_1}$ or $x_{k_2}$. 

If $d=4$ and $v_l$ is parallel to $x_{k_1}$ (or $x_{k_2}$ respectively), then $n \equiv 2 \mmod 4$ and in particular $n \geq 10$.
In this case one cylinder $c$ in direction $v_l$ contains the gluing along $x_{k_2}$ (or $x_{k_1}$ respectively). This cylinder has two preimage cylinders, both with inverse modulus $2 \lambda_n$ and cylinder number $\frac{n-2}{4}-1$ i.e. it has height $h = 2 \cos( \frac{n-8}{2} \cdot \frac{\pi}{n}) \sin(\frac{\pi}{n})$ (see equation \ref{eq:widthAndHeightEven} on page \pageref{eq:widthAndHeightEven}).
 The two cylinders in the horizontal decomposition of $Y_{n,4}$ with inverse modulus $2 \lambda_n$ are preimages of cylinder $\frac{n-2}{4}$, so they have height $\tilde{h} = 2 \cos( \frac{n-4}{2} \cdot \frac{\pi}{n}) \sin(\frac{\pi}{n})$. Since 
 $ \tilde{h} - h = 2 \sin(\frac{\pi}{n})(\cos( \frac{n-4}{2} \cdot \frac{\pi}{n}) - \cos( \frac{n-8}{2} \cdot \frac{\pi}{n})) = -4 \sin(\frac{n-6}{2} \cdot \frac{\pi}{n}) \sin^2\!(\frac{\pi}{n}) \neq 0$,
no affine map with derivative ${R_n}^{2l}$ exists on $Y_{n,4}$.

If $n \equiv 0 \mmod 4$ and $l = \frac{n}{4}$, then $v_l$ is vertical and the cylinder decomposition of $Y_{n,d}$ in direction $v_l$ has a cylinder with height $h = 2 \cos(\frac{\pi}{n}) \sin(\frac{\pi}{n})$ and inverse modulus $d \lambda_n$ or two cylinders with inverse modulus $\frac{d}{2} \lambda_n$, respectively (depending on $d$ being odd or even). 
The cylinder with inverse modulus $d \lambda_n$ or $\frac{d}{2} \lambda_n$ in the horizontal decomposition of $Y_{n,d}$ has height $\tilde{h} = 2 \cos( \frac{n-2}{2} \cdot \frac{\pi}{n}) \sin(\frac{\pi}{n})$ which is 
different from $h$. 
Hence we may conclude that no affine map with derivative ${R_n}^\frac{n}{2}$ exists on $Y_{n,d}$.
\end{proof}

We deduce the following Corollary on the Teichm\"uller curve to $Y_{n,d}$.

\begin{corollary}
\label{cor:CurveEven}
 $\HH / \Gamma_n$ has genus $0$ and $\frac{n+2}{2}$ cusps.
\end{corollary}
\begin{proof}
\begin{figure}[htbp]
\centering
\subfigure[$\Gamma(X_n) = \langle T_n, {R_n}^2 \rangle$]{
\label{fig:FundX6}
\centering
 \includegraphics[width=0.3\textwidth]{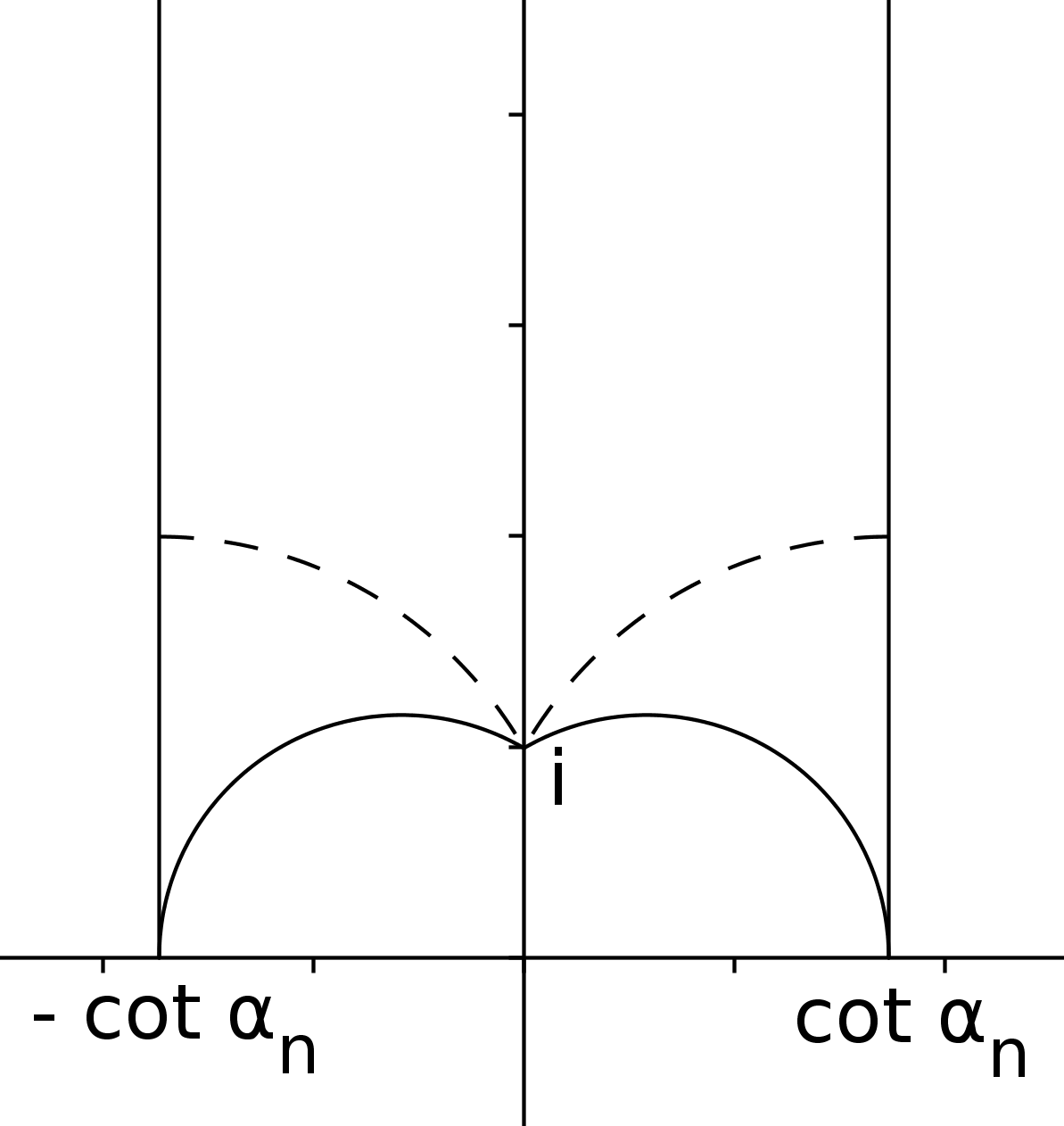}
}
\subfigure[$\Gamma_8$]{
\label{fig:FundY8}
\centering
 \includegraphics{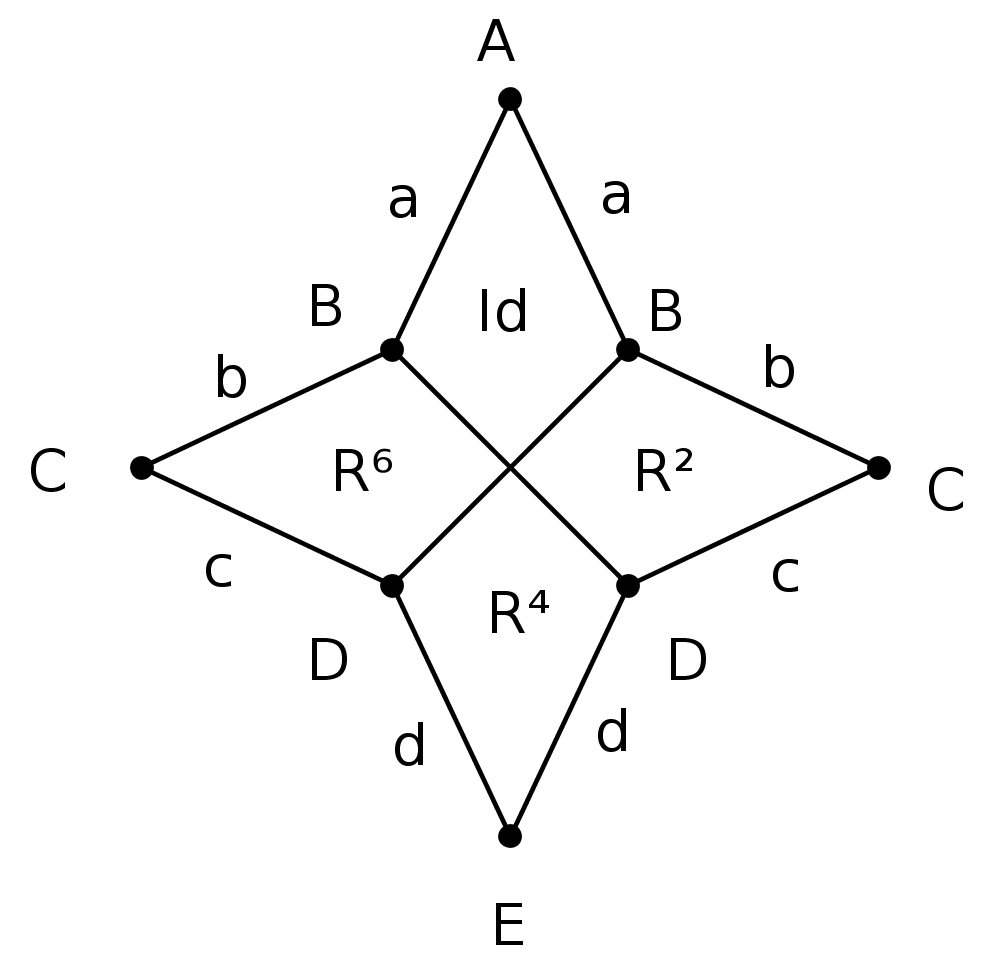}
}
 \caption{Fundamental regions}
 \label{fig:FundamentalbereichX6}
\end{figure}
 The Veech group $\Gamma(X_n) = \langle T_n, {R_n}^{2} \rangle$ has index two in the Hecke group $\langle T_n, R_n \rangle$. In particular it has a fundamental region $F$ as shown in
 Figure \ref{fig:FundX6} with $2$ cusps, $\{\infty\}$ and $\{- \cot \alpha_n, \cot \alpha_n\}$ with $\alpha_n = \frac{\pi}{n}$. A fundamental region of $\Gamma_n$ is $\widetilde{F} = F \cup {R_n}^2(F) \cup \dots \cup {R_n}^{n-2}(F)$. The vertices of $\widetilde{F}$ in $\R \cup \{\infty\}$ are $\{ \infty\}$ and $\cot(r \, \alpha_n)$ with $r \in \{ 1, \dots, n-1 \}$. The matrices ${R_n}^{2l} \, {T_n}^2 \, {R_n}^{-2l}$ and ${R_n}^{2l} \, ({T_n}^{-1}{R_n}^2)^2 \, {R_n}^{-2l}$ identify $\frac{n-2}{2}$ pairs of vertices (see Figure \ref{fig:FundY8} for $\Gamma_8$), so $\HH / \Gamma_n$ has $\frac{n+2}{2}$ cusps. Using the Euler characteristic, it can be seen that $\HH / \Gamma_n$ has genus $0$.
\end{proof}

\section{Infinite translation surfaces.}
\label{sec:inf}

In this section we study the limit of our covering families 
and show that their Veech groups are lattices in $\SL$. An \textit{infinite translation surface} $X^\infty$ can be obtained by gluing countably infinitely many polygons via identification of edge pairs using translations. We only consider the case where countably infinitely many copies of one polygon are glued together. 
Then, if infinitely many copies are glued at one vertex, an infinite angle singularity arises.
Let $\Sigma^\infty$ be the set of infinite angle singularities, then $X^\infty \setminus \Sigma^\infty$ is a Riemann surface and $X^\infty$ is the metric completion of $X^\infty \setminus \Sigma^\infty$. So we get a so called \textit{tame flat surface}.
For more details on infinite translation surfaces see e.g. \cite{Val09}.\\

As in the finite case, let $n \geq 5$, $n \neq 6$. To simplify notation, let $X^\ast = X \setminus \Sigma(X)$ for a finite or tame infinite translation surface $X$. Our series of translation coverings give rise to the infinite translation surface $Y_{n,\infty}$, defined by the monodromy
$$m_{n,\infty}: \left\{\begin{array}{rcl}
                 \pi_1(X_n^\ast) & \longrightarrow & \textrm{Sym}(\Z)\\
		 x_i & \mapsto & \id \quad \textrm{, for } i \notin \{k_1,k_2\}\\
		 x_{k_1} & \mapsto & \sigma_{k_1}\\
		 x_{k_2} & \mapsto & \sigma_{k_2}
                \end{array}\right.$$
where $\sigma_{k_1}: l \mapsto \left\{\begin{array}{rl}
				   l+1 & \textrm{, } l \textrm{ even}\\
				   l-1 & \textrm{, } l \textrm{ odd}
				  \end{array}\right.$,
$\sigma_{k_2}: l \mapsto \left\{\begin{array}{rl}
				   l-1 & \textrm{, } l \textrm{ even}\\
				   l+1 & \textrm{, } l \textrm{ odd}
				  \end{array}\right. \, ,$  
$$k_1 = \left\{\begin{array}{ll}
  \frac{n-1}{2} & \textrm{, } n \textrm{ odd}\\
  \frac{n}{4} -1 & \textrm{, } n \equiv 0 \mmod 4\\
  \frac{n-2}{4} -1 & \textrm{, } n \equiv 2 \mmod 4
  \end{array}\right. 
  \quad \textrm{ and } \quad 
  k_2 = \left\{\begin{array}{ll}
  \frac{n+1}{2} & \textrm{, } n \textrm{ odd}\\
  \frac{n}{4} & \textrm{, } n \equiv 0 \mmod 4\\
  \frac{n-2}{4} + 1 & \textrm{, } n \equiv 2 \mmod 4
  \end{array}\right. \, .$$

\begin{remark}
 The surface $Y_{n,\infty}$ has $4$ infinite angle singularities.
\end{remark}
\begin{proof}
 For odd $n$ and even $n$, $n \equiv 0 \mmod 4$, a simple clockwise path around the singularity in $X_n$ is given by 
$$\begin{array}{lc}
   p = x_0 {x_1}^{-1} x_2 {x_3}^{-1} x_4 \dots x_{n-3} {x_{n-2}}^{-1} {x_0}^{-1} x_1 {x_2}^{-1} x_3 \dots {x_{n-3}}^{-1} x_{n-2}
  &\textrm{or}\\
 p = x_0 {x_1}^{-1} x_2 {x_3}^{-1} x_4 \dots x_{\frac{n}{2}-2} {x_{\frac{n}{2}-1}}^{-1} {x_0}^{-1} x_1 {x_2}^{-1} x_3 \dots {x_{\frac{n}{2}-2}}^{-1} x_{\frac{n}{2}-1} & ,
  \end{array}$$
respectively.
Since the monodromy of $x_i$ is the identity iff $i \notin \{k_1, k_2\}$, $\sigma_{k_1} = {\sigma_{k_1}}^{-1}$ and $\sigma_{k_2} = {\sigma_{k_2}}^{-1}$, we have
$$ (m_{n,\infty} (p))\,(l) = (\sigma_{k_2} \circ \sigma_{k_1})^2 (l)= \left\{ \begin{array}{ll}
                               l+4 & , \,l \textrm{ even}\\
			       l-4 & , \,l \textrm{ odd}\\
                              \end{array}\right. \;.$$
For even $n$, $n \equiv 2 \mmod 4$, two simple clockwise paths $p_1$ and $p_2$ around the two singularities in $X_n$ are given by
$
p_1 = x_0 {x_1}^{-1} x_2 {x_3}^{-1} x_4 \dots {x_{\frac{n}{2}-2}}^{-1} x_{\frac{n}{2}-1}$ and
$p_2 = x_1 {x_2}^{-1} x_3 \dots x_{\frac{n}{2}-2} {x_{\frac{n}{2}-1}}^{-1} {x_0}^{-1} 
  \,.$
It follows
$$
  (m_{n,\infty} (p_1))\, (l)= 
     (m_{n,\infty} (p_2))\,(l) =
      \sigma_{k_2} ( \sigma_{k_1} (l))
      = \left\{ \begin{array}{ll}
                               l+2 & , \,l \textrm{ even}\\
			       l-2 & , \,l \textrm{ odd}\\
                              \end{array}\right.
 \;.$$
\end{proof}

\begin{remark}
 The covering $p_{n,\infty}: Y_{n,\infty} \to X_n$ is not a $\Z$-covering, but $Y_{n,\infty}$ is a $\Z$-covering of $Y_{n,2}$:
 A basis of $\pi_1(Y_{n,2}^\ast)$ is $B = \{ x_{k_2} {x_{k_1}}^{-1},{x_{k_1}}^2, x_{k_1} x_{k_2} \} \cup \{x_i, \, x_{k_1} x_i {x_{k_1}}^{-1} \mid i \notin\{k_1, k_2\} \, \}$ and $\pi_1(Y_{n,\infty}^\ast) \subseteq \pi_1(Y_{n,2}^\ast)$.
 Every copy of $Y_{n,2}$ in $Y_{n,\infty}$ consists of two copies of $X_n$ with numbers $2l$ and $2l+1$.
 Since $\sigma_{k_2} \sigma_{k_1} (2l) = 2l+2$ and ${\sigma_{k_1}}^{-1} \sigma_{k_2} = (\sigma_{k_2} \sigma_{k_1})^{-1}$, the image of $\pi_1(Y_{n,2}^\ast)$ under $m_{n,\infty}$ restricted to the even numbers defines a transitive permutation group on $2 \Z$ isomorphic to $\Z$. It follows that
 $$\tilde{m} = {m_{n,\infty}}\arrowvert_{\pi_1(Y_{n,2}^\ast)} = \left\{
 \begin{array}{rcl}
  \pi_1(Y_{n,2}^\ast) & \to & \textrm{Sym}(2 \Z) \cong \textrm{Sym}(\Z)\\
  x_{k_1} x_{k_2} & \mapsto & (2l \mapsto 2l+2) \\ 
  x_{k_2} {x_{k_1}}^{-1} & \mapsto & (2l \mapsto 2l-2)\\ 
  w & \mapsto & \id \quad \quad ,\, w \in B \setminus\{x_{k_1} x_{k_2}, x_{k_2} {x_{k_1}}^{-1}\}
 \end{array}
 \right.$$
defines a normal covering $\tilde{p}: Y_{n,\infty} \to Y_{n,2}$ with Galois group $\Z$ and $p_{n,\infty} = p_{n,2} \circ \tilde{p}$. 

In \cite{HoWe09} $\Z$-covers are defined by non-zero elements in the relative homology $H_1(X, \Sigma(X); \Z)$. In our case, a representative for the element $w \in H_1(Y_{n,2}, \Sigma(Y_{n,2}); \Z)$, defining $Y_{n,\infty}$, consists of two reverse paths on the edges $x_{k_2}$ in copy $0$ and copy $1$. For $n=8$ this can be seen in Figure \ref{fig:relHom}. 
\begin{figure}[htb]
 \centering
 \includegraphics{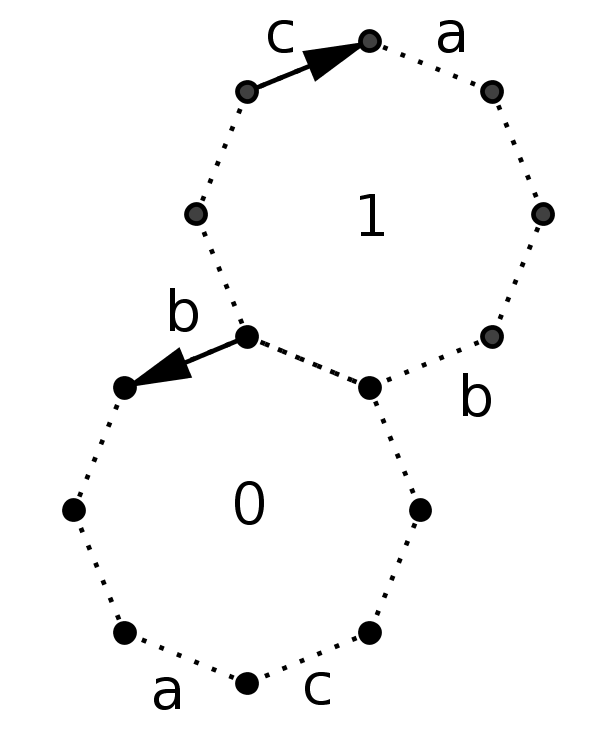}
 \caption{$w \in H_1(Y_{8,2}, \Sigma(Y_{8,2}); \Z)$, defining $Y_{8,\infty}$}
 \label{fig:relHom}
\end{figure}
The holonomy $hol(w) = 0$, so according to \cite{HoWe09} the cover is recurrent, i.e. the straight line flow on $Y_{n,\infty}$ is recurrent for almost every direction $\theta$.
\end{remark}

We want to determine the Veech group of $Y_{n,\infty}$ and start by arguing, why $\Gamma(Y_{n, d})$ is contained in $\Gamma(X_n)$, even for $d = \infty$. 
For this we use the developments of the saddle connections of a translation surface $(X,\omega)$ (see \cite{Vor96} Chapter 3.1). For every saddle connection there is a chart containing the saddle connection without end points. Thus the saddle connection defines an interval in $\R^2$, i.e.\ a vector and its additive inverse in $\R^2$, not depending on the chosen chart. The set of all theses vectors 
is called $SC(X)$. 

\begin{lemma}
\label{Lemma:inf}
 For $n \geq 8$, the Veech group $\Gamma(Y_{n, d})$ of $Y_{n,d}$ is contained in $\Gamma(X_n)$, even for $d = \infty$. 
\end{lemma}

\begin{proof}
The covering map $p_{n,d}: Y_{n,d} \to X_n$ is locally a translation, so every saddle connection in $Y_{n, d}$ is mapped one to one to a saddle connection in $X_n$ and every saddle connection in $X_n$ can be lifted to a saddle connection in $Y_{n,d}$. It immediately follows $SC(X_n) = SC(Y_{n,d})$. This idea, at least for finite $d$, and the following statement can be found in \cite{Vor96}.
Affine maps preserve the set of saddle connections. We conclude that $\Gamma(Y_{n, d})$ is contained in the stabiliser $\textrm{Stab}(SC(Y_{n,d}))$ of $SC(Y_{n,d})$ and $\Gamma(X_n) \subseteq \textrm{Stab}(SC(X_n))$, respectively.
 According to \cite{Vor96} Proposition 3.1, $SC(X_n)$ has no limit points, hence $\textrm{Stab}(SC(X_n))$ is a discrete group. 
 
 For odd $n$, the Veech group $\Gamma(X_n)$ is a Hecke group. In particular it is a maximal discrete subgroup of $\SL$ and $\Gamma(X_n) \subset \textrm{Stab}(SC(X_n))$ implies $\Gamma(X_n) = \textrm{Stab}(SC(X_n))$ which proves the statement for $n$ odd. 
 
For even $n$, $\Gamma(X_n)$ is a triangle group with signature $(\frac{n}{2}, \infty, \infty)$. According to \cite{Si72} Theorem $1$ and Theorem $2$, the only possible discrete groups properly containing $\Gamma(X_n)$ are the triangle groups with signature $(n, 2, \infty)$.
Let $G$ denote the triangle group $(n, 2, \infty)$, containing $\Gamma(X_n)$. Then $G = \langle T_n, R_n \rangle$ and $[G:\Gamma(X_n)] = 2$. The shortest saddle connections in $X_n$ are the edges of the regular $n$-gon. 
Their corresponding elements in $SC(X_n)$ are 
$$\begin{array}{rcl}
   v_l &=& {R_n}^{2(l+1)} \cdot \begin{pmatrix} 
                1 \\ 0
               \end{pmatrix}
      - {R_n}^{2l} \cdot \begin{pmatrix}
                1 \\ 0
               \end{pmatrix}
= \begin{pmatrix}
\cos (2(l+1) \frac{\pi}{n}) - \cos (2l \frac{\pi}{n})\\
\sin (2(l+1) \frac{\pi}{n})- \sin (2l \frac{\pi}{n})
\end{pmatrix} \\
&=& \begin{pmatrix}
   - 2 \sin ((2l+1) \frac{\pi}{n} ) \sin(\frac{\pi}{n})\\
   2 \cos((2l+1) \frac{\pi}{n}) \sin(\frac{\pi}{n})
  \end{pmatrix} \,.
  \end{array} $$
The rotation $R_n$ doesn't change the length of a saddle connection, but 
$$R_n \cdot v_l = {R_n}^{2(l+1) + 1} \cdot \begin{pmatrix} 
                1 \\ 0
               \end{pmatrix}
      - {R_n}^{2l +1 } \cdot \begin{pmatrix}
                1 \\ 0
               \end{pmatrix}
= \begin{pmatrix}
   - 2 \sin ((2l+2) \frac{\pi}{n} ) \sin(\frac{\pi}{n})\\
   2 \cos((2l+2) \frac{\pi}{n}) \sin(\frac{\pi}{n})
  \end{pmatrix} \neq v_j
$$
for all $j \in \{0, \dots, n-1\}$. So $R_n \notin \textrm{Stab}(SC(X_n))$ and $\textrm{Stab}(SC(X_n)) = \Gamma(X_n)$. 
 We conclude $\Gamma(Y_{n,d}) \subset \textrm{Stab}(SC(Y_{n,d})) = \textrm{Stab}(SC(X_n)) = \Gamma(X_n)$.
\end{proof}

Now we can calculate the Veech group of the infinite covering surface by reproducing the arguments of the finite case.

\begin{ThmInf}
\label{prop:VeechGroupInf}
 For $n \geq 5$, $n \neq 6$, the Veech group $\Gamma(Y_{n, \infty})$ of $Y_{n,\infty}$ is $\Gamma_n$. In particular $Y_{n,\infty}$ is an infinite translation surface with a lattice Veech group.
\end{ThmInf}
\begin{proof}

We reconsider the steps 1 to 4 of the proof in the finite case. In step~1 we can use identical arguments to show that ${R_n}^l \, {T_n}^2 \, {R_n}^{-l}$ (if $n$ is odd) 
or that ${R_n}^{2l} \, {T_n}^2 \, {R_n}^{-2l}$ 
and ${R_n}^{2l} \,({T_n}^{-1}{R_n}^2)^2 \, {R_n}^{-2l}$ 
(if $n$ is even) 
are contained in $\Gamma(Y_{n,\infty})$. Once again the proof for ${R_n}^\frac{n}{2} \, {T_n}^2 \, {R_n}^{-\frac{n}{2}}$ if $n \equiv 0 \mmod 4$ and for ${R_n}^{\frac{n+2}{2}} \,({T_n}^{-1}{R_n}^2)^2 \, {R_n}^{-\frac{n+2}{2}}$ if $n \equiv 2 \mmod 4$  rely on $\sigma_T$ (as defined below).
 
 To prove that $I, R_n, \dots, {R_n}^{n-1}$ (or $I, {R_n}^2, \dots, {R_n}^{n-2}$ respectively) lie in different cosets of $\Gamma(Y_{n,\infty}) \bs \Gamma(X_n)$ in step 2, we compute the monodromy of the core curve $p$ of cylinder 
 $$k = \left\{\begin{array}{ll}
	 \frac{n-1}{2} &, \, n \textrm{ odd}\\
	 \frac{n}{4} &, \, n \equiv 0 \mmod 4\\
	 \frac{n-2}{4} &, \, n \equiv 2 \mmod 4
      \end{array}\right.$$ 
of the horizontal cylinder decomposition of $X_n$:
 $$m_{n,\infty}(p) = m_{n,\infty} (x_{k_1} {x_{k_2}}^{-1}) = {\sigma_{k_2}}^{-1} \circ \sigma_{k_1}
 = \left\{ \begin{array}{ll}
	    i \mapsto i+2 & ,\, i \textrm{ even}\\
	    i \mapsto i-2 & , \, i \textrm{ odd}
            \end{array} \right.$$
So cylinder $k$ has two preimages of infinite length in $Y_{n,\infty}$. Recall that no cylinder of $X_n$ in the decomposition in direction $v_l$ 
contains both $x_{k_1}$ and $x_{k_2}$ if $n$ is odd or if $n$ is even and $l \neq \frac{n}{4}$. This implies that $Y_{n,\infty}$ contains no infinite cylinder (i.e. no strips) in these directions. 
As in the finite case $n \equiv 0 \mmod 4$ and $l = \frac{n}{4}$ is an exception. There the cylinder decomposition in direction $v_l$ contains two infinite cylinders. But their height differs from the height of the infinite cylinders in the horizontal decomposition.
We conclude that there is no affine map with derivative ${R_n}^l$ (or ${R_n}^{2l}$) on $Y_{n, \infty}$.

The proof of $T_n \in \Gamma(Y_{n, \infty})$ is similar to the finite case when $d$ is even. We consider the two infinite length preimages of cylinder $k$ of the horizontal cylinder decomposition. The maps $\suc$ and $\pred$ are defined as in the finite case and again we need to prove the existence of a permutation $\sigma_T \in \textrm{Sym}(\Z)$ fulfilling the two properties:
 $\sigma_T(\suc(i)) = \suc(\sigma_T(i))$ and
$m_{n,\infty}(x_{k_1}) (\sigma_T(i)) = \sigma_T(m_{n,\infty}(x_{k_2})(i)) \,. $
Let 
$$\sigma_T = \left\{
\begin{array}{ll} 
i \mapsto i & , \, i \textrm{ even}\\
i \mapsto i+2 & , \, i \textrm{ odd}
\end{array} \right. \,.$$
The map $\suc = m_{n,\infty} (x_{k_1} {x_{k_2}}^{-1})$ was calculated above. It remains to check the two conditions on $\sigma_T$:
$$\sigma_T(\suc(i)) =  
\left\{ \begin{array}{ll}
         i+2 & , \, i \textrm{ even}\\
	 i & , \, i \textrm{ odd}
        \end{array} \right.
= \suc(\sigma_T(i))$$
$$m_{n,\infty}(x_{k_1}) (\sigma_T(i))
= i+1 
= \sigma_T(m_{n,\infty}(x_{k_2})(i))$$
Finally the proof for $-I \in \Gamma(Y_{n,\infty})$ in step 4 works as in the finite case. Lemma~\ref{Lemma:inf} and Lemma~\ref{Lemma:VeechGroupGen} complete the proof.
\end{proof}

\bibliographystyle{amsalpha} 
\bibliography{paperMathe.bib}

\providecommand{\bysame}{\leavevmode\hbox to3em{\hrulefill}\thinspace}
\providecommand{\MR}{\relax\ifhmode\unskip\space\fi MR }
\providecommand{\MRhref}[2]{%
  \href{http://www.ams.org/mathscinet-getitem?mr=#1}{#2}
}
\providecommand{\href}[2]{#2}
\begin{thebibliography}{McM06}

\bibitem[Fre08]{Fr08}
Myriam Freidinger, \emph{{Stabilisatorgruppen in Aut($F_z$) und Veechgruppen
  von \"Uberlagerungen}}, diploma thesis, Universit\"at Karlsruhe, 2008.

\bibitem[GJ00]{GJ00}
E.~Gutkin and C.~Judge, \emph{{Affine mappings of translation surfaces:
  Geometry and arithmetic}}, Duke Math. J. \textbf{103} (2000), no.~2,
  191--213.

\bibitem[Hoo08]{Ho08}
P.~Hooper, \emph{{Dynamics on an infinite surface with the lattice property}},
  2008, arXiv:0802.0189v1.

\bibitem[HS01]{HS01}
P.~Hubert and T.A. Schmidt, \emph{{Invariants of translation surfaces}},
  Annales de l'institut Fourier \textbf{51} (2001), no.~2, 461--495.

\bibitem[HS07]{HSch07}
F.~Herrlich and G.~Schmith\"usen, \emph{{On the boundary of Teichm\"uller disks
  in Teichm\"uller and in Schottky space}}, Handbook of Teichm\"uller Theory
  \textbf{I} (2007), 293--349.

\bibitem[HS10]{SchHu09}
P.~Hubert and G.~Schmith\"usen, \emph{{Infinite translation surfaces with
  infinitely generated Veech groups}}, J. Mod. Dyn. \textbf{4} (2010), no.~4,
  715--732.

\bibitem[HW09]{HoWe09}
P.~Hooper and B.~Weiss, \emph{{Generalized staircases: recurrence and
  symmetry}}, 2009, arXiv:0905.3736v1.

\bibitem[Joh97]{Joh97}
D.L. Johnson, \emph{Presentations of groups}, second ed., Cambridge University
  Press, 1997.

\bibitem[Kat92]{Kat92}
S.~Katok, \emph{{Fuchsian Groups}}, University of Chicago Press, 1992.

\bibitem[M\"06]{Moe06}
M.~M\"oller, \emph{{Periodic points on Veech surfaces and the Mordell-Weil
  group over a Teichm\"uller curve}}, Invent. Math. \textbf{165} (2006), no.~3,
  633--649.

\bibitem[McM06]{McM06}
C.T. McMullen, \emph{{Prym varieties and Teichm\"uller curves}}, Duke Math. J.
  \textbf{133} (2006), no.~3, 569--590.

\bibitem[Mir95]{Mir95}
R.~Miranda, \emph{{Algebraic Curves and Riemann Surfaces}}, American
  Mathematical Society, 1995.

\bibitem[Sch04]{Sch04}
G.~Schmith\"usen, \emph{{An algorithm for finding the Veech group of an
  origami}}, Experimental Mathematics \textbf{13} (2004), no.~4, 459--472.

\bibitem[Sch06]{Sch06}
\bysame, \emph{{Examples for Veech groups of origamis}}, Proceedings of the III
  Iberoamerican Congress on Geometry. In: The Geometry of Riemann Surfaces and
  Abelian Varieties. Contemp. Math. \textbf{397} (2006), 193--206.

\bibitem[Sin72]{Si72}
D.~Singerman, \emph{{Finitely maximal Fuchsian groups}}, J. London Mathematical
  Society \textbf{s2-6} (1972), 29--38.

\bibitem[Val09]{Val09}
F.~Valdez, \emph{{Veech groups, irrational billiards and stable abelian
  differentials}}, 2009, arXiv:0905.1591v2.

\bibitem[Vee89]{Vee89}
W.A. Veech, \emph{{Teichm\"uller curves in moduli space, Eisenstein series and
  an application to triangular billiards}}, Invent. math. \textbf{97} (1989),
  553--583.

\bibitem[Vor96]{Vor96}
Y.B. Vorobets, \emph{{Planar structures and billiards in rational polygons: the
  Veech alternative}}, Russ. Math. Surv. \textbf{51} (1996), 779--817.

\bibitem[ZK75]{ZK75}
A.N. Zemlyakov and A.B. Katok, \emph{Topological transitivity of billiards in
  polygons}, Math. Notes \textbf{18} (1975), 760--764.

\end{thebibliography}

\end{document}